\documentclass[12pt]{article}
\usepackage[latin1]{inputenc}
\usepackage[english]{babel}
\usepackage[T1]{fontenc}
\usepackage{lmodern}
\usepackage{mathrsfs}
\usepackage{mathtools}
\usepackage{amsmath,amsthm,amssymb}
\usepackage[mathlines]{lineno}
\usepackage{enumerate}
\usepackage{enumitem}
\usepackage[colorlinks=true,citecolor=black,linkcolor=black,urlcolor=blue]{hyperref}
\usepackage[top=.9in, bottom=.9in, left=.5in , right=.5in]{geometry}
\usepackage{comment}
\mathtoolsset{showonlyrefs}
\newcommand{\E}[2][]{\ensuremath{\mathbb{E}_{#1} \left[#2 \right]}}
\newcommand{\Prob}[2][]{\ensuremath{\mathbb{P}_{#1} \left(#2 \right)}}
\newcommand{\dd}{\mathrm d}

\newcommand{\N}{\mathbb{N}}

\def\supp{{\rm Supp}}
\newcommand{\pairing}[1]{\left \langle #1 \right \rangle }

\newcommand{\R}{\mathbb{R}}

\newcommand{\Ecal}{\mathcal{E}}

\newcommand{\massminus}[1]{\langle m, #1 \rangle}
\newcommand{\mass}[1]{\left\langle m \mathbf{1}_{m \geq 1}, #1 \right\rangle}
\newcommand{\init}{\mu}

\newcommand{\ME}{\mathcal{M}_{+}(E)}

\newcommand{\eps}{\varepsilon}

\newtheorem{thm}{Theorem}[section]
\newtheorem{lemma}[thm]{Lemma}

\newtheorem{clm}{Claim}[thm]
\newtheorem{prop}[thm]{Proposition}
\newtheorem{cor}[thm]{Corollary}
\newtheorem{assumption}[thm]{Assumption}
\theoremstyle{definition}
\newtheorem{rmq}{Remark}[section]
\newtheorem{example}[rmq]{Example}
\newtheorem{examples}[rmq]{Examples}
\newtheorem{defn}{Definition}[section]

\begin{document}

\title{Convergence of cluster coagulation dynamics}
\author{L. Andreis\footnote{Politecnico di Milano, Dipartimento di matematica,
Piazza Leonardo da Vinci, 32,
20133 Milano.}, T. Iyer\footnote{Weierstrass Institute for Applied Analysis and Stochastics, Mohrenstrasse 39, 10117 Berlin, Germany.},  E. Magnanini\footnotemark[2].}
\maketitle
\abstract{We study hydrodynamic limits of the cluster coagulation model; a coagulation model introduced by Norris [\textit{Comm. Math. Phys.}, 209(2):407-435 (2000)]. In this process, pairs of particles $x,y$ in a measure space $E$, merge to form a single new particle $z$ according to a transition kernel $K(x, y, \dd z)$, in such a manner that a quantity, one may regard as the total \emph{mass} of the system, is conserved. 
This model is general enough to incorporate various inhomogeneities in the evolution of clusters, for example, their shape, or their location in space. 
We derive sufficient criteria  for trajectories associated with this process to concentrate among solutions of a generalisation of the \emph{Flory equation}, and, in some special cases, by means of a uniqueness result for solutions of this equation, prove a weak law of large numbers.  This multi-type Flory equation is associated with \emph{conserved quantities} associated with the process, which may encode different information to conservation of mass (for example, conservation of centre of mass in spatial models). We also apply criteria for \emph{gelation} in this process to derive sufficient criteria for this equation to exhibit \emph{gelling} solutions. When this occurs, this multi-type Flory equation encodes, via the associated conserved property, the interaction between the \emph{gel} and the finite size \emph{sol} particles.}

\maketitle
\section{Introduction} \label{intro}
Coagulation models are prevalent across numerous disciplines, spanning from physical chemistry (to describe polymer formation), to astrophysics, where they simulate galaxy formation. A natural coagulation model involves particles moving in space, with larger particles exhibiting slower movement compared to smaller ones.
A classical `mean field' description of this process via Markovian dynamics is known as the \emph{Marcus-Lushnikov model}~\cite{Marcus68, Gil72, Lushnikov78}. Here, pairs of particles of masses $x$ and $y$ merge at a rate $\bar{K}(x,y)$ (for an appropriate function $\bar{K}(x,y)$). The re-scaled limiting behaviour of the trajectories of empirical measures encoding particle masses in this model, is generally expected to follow a system of measure-valued differential equations known as the \emph{Smoluchowski} (or \emph{Flory}) equations. This re-scaling, which involves re-scaling time with the system size (to `slow time down'), and looking at a fixed \emph{volume} (the whole space $E$), is known as taking a \emph{hydrodynamic limit}. 

A popular question with regards to coagulation processes, concerns whether or not one sees the occurrence of \emph{macroscopic} or \emph{giant} particles by some time $t > 0$; a phenomenon known as \emph{gelation}. \emph{Gelation} is generally defined in terms of whether or not a solution of the Smoluchowski (or Flory) equation fails to `conserve mass', which means, intuitively, that mass is lost to `infinite' particles. Meanwhile, the term stochastic gelation, generally refers to the formation of `large particles' in a bounded amount of time in the Marcus--Lushnikov process. When the trajectories of the Marcus--Lushnikov process are concentrated on solutions of the Smoluchowski equation, Jeon showed that these two notions of gelation are equivalent~\cite[Theorem~5]{jeon98}. However, in certain regimes, it may be the case, that, once gelation occurs, the `gel' or macroscopic particles, interact with the microscopic ones or `sol', in such a way that the trajectories of the Marcus--Lushnikov process only concentrate around solutions of the Smoluchowski equation for a short time before gelation. In such regimes, a correction term encoding this `sol-gel' interaction is required in the Smoluchowkski equation, this is known as the Flory equation. A weak law of large numbers of the Marcus--Lushnikov process to the Flory equation in a particular case was first proved by Norris~\cite{norris-cluster-coag}, and later, concentration of trajectories around solutions of this equation were proved in \cite{fournier-giet-04, rezakhanlou2013}. 

A limitation of the classical Marcus--Lushnikov model is that the dynamics only include information of the masses of particles, and not, for example, other inhomogeneities such as their location in space. 
In~\cite{norris-cluster-coag}, Norris introduced the \emph{cluster coagulation model}, which allows one to incorporate these features rather generally, (see also~\cite{Norris} for a variant where clusters move diffusively). In a particular, limited regime of this model, Norris proved a weak law of large numbers for the dynamics of this model, with exponential rates of convergence~\cite[Theorem~4.2 \&~4.3]{norris-cluster-coag}. It should be noted that the cluster coagulation model is general enough to encompass models such as those studied by Jacquot~\cite{jacquot-10} and later work of Heydecker and Patterson~\cite{heydecker2019bilinear}. In ~\cite{jacquot-10}, Jacquot proves a weak law of large numbers for the `historical trees' encoding histories of clusters in the Marcus-Lushnikov process is proved; whilst in~\cite{heydecker2019bilinear}, Heydecker and Patterson prove a weak law of large numbers in the special case of `bilinear' kernels. For more examples of cluster coagulation processes we refer the reader to~\cite[Section~3.1]{AnIyMa24}. Also note that more recently, from the perspective of analysis, a strain of research has focused on the study of multi-component generalisations of the Smoluchowski coagulation equations, where the mass variable is substituted by a variable in a  $d$-dimensional Euclidean space, and there is an additional source term corresponding to the system not being in equilibrium (see for example, \cite{FeLuNoVe21, FeLuNoVe22, Thr23} and reference therein). In general, without the source term, these equations are particular instances of the Smoluchowski equation associated with the cluster coagulation model.

In recent work~\cite{AnIyMa24}, the authors have proved sufficient criteria for stochastic gelation to occur in the cluster coagulation process. In this paper, we derive criteria for concentration of trajectories associated with the cluster coagulation model around solutions of a multi-type Flory equation, generalising the classical Flory equation. We also provide criteria for solutions of this equation to be unique, hence for sufficient condition for a weak law of large numbers associated with these trajectories. As a result of the aforementioned device due to Jeon~\cite[Theorem~5]{jeon98} (reformulated in~\cite[Theorem~1.1]{AnIyMa24}) this concentration also leads to sufficient criteria for the associated multi-type Flory equation to exhibit gelling solutions.

\subsection{Overview on our contribution}
The main novelty of this paper involves the idea of 
a \emph{conserved quantity} appearing in a generalisation of the Flory equation; corresponding to an invariant associated with the cluster coagulation process. From the perspective of applications, this quantity may correspond to the total mass of the system, or, depending on the setting, for example, the `centre of mass' or `momentum'. 
\begin{enumerate}
    \item In Definition~\ref{weak_sol} we define a \emph{multi-type Flory equation}, expected to encode the limiting behaviour of the cluster coagulation process, when there are `conserved quantities'. Then, in Theorem~\ref{spat_flory} we provide sufficient criteria for tightness of trajectories associated with the cluster coagulation process, and criteria for limit points associated with the process to concentrate on solutions of this equation, proving, in particular, the existence of solutions.
    \item In Theorem~\ref{thm:uniqueness} we obtain a uniqueness result for this multi-type Flory equation  in the case when the kernel is `eventually conservative' (see Definition~\ref{cons_q}). Such a uniqueness result implies a \emph{law of large numbers} for the paths of the stochastic cluster coagulation process in Corollary~\ref{cor:LLN}.
    \item In Corollaries~\ref{cor:class-gel-sol} and~\ref{cor:gel-2}, we state, without proof, conditions for solutions of the multi-type Flory equation to have gelling solutions. These results rely on previous results concerning stochastic gelation from the aforementioned companion paper~\cite{AnIyMa24}. 
\end{enumerate}
 The notion of conserved quantities allows us to go beyond the setting of \cite{norris-coag-99}, and to prove existence of solutions for the limiting equation under weaker assumptions. The `eventually conservative' property is a natural extension of the \emph{eventual multiplicativity} property introduced in~\cite{norris-cluster-coag}. The approach we use, using weak-compactness, and martingale techniques, whilst well-established, allow for relatively few assumptions on the underlying space. We only assume that the clusters take values in a $\sigma$-compact metric space $E$. This means that our results encapsulate existing generalisations of the Marcus-Lushnikov model in the literature, including those of~\cite{heydecker2019bilinear, jacquot-10}. 
 
\subsection{Structure}
The rest of the paper is structured as follows. 
\begin{enumerate}
    \item In Section~\ref{intro-subsec-model} we introduce the model, terminology and global assumptions in Assumption~\ref{ass:global}.
    \item  In Section~\ref{sec:conc-trajectories} we include our main definitions, including the definition of a generalised Flory equation with a \emph{conserved quantity} in Definition~\ref{weak_sol}. This section also includes our main result, Theorem~\ref{spat_flory}, concerning the hydrodynamic limit of the process.
    \item  In Section~\ref{sec:unique} we state a uniqueness result (Theorem~\ref{thm:uniqueness}) for the solutions of the multi-type Flory equation in the case of \emph{eventually conservative} kernels.
    \item In Section~\ref{sec:existence_gelling} we draw connections with the related paper~\cite{AnIyMa24} on gelation for these models to state results concerning existence of gelling solutions for the multi-type Flory equation. As these results are an immediate consequence of results appearing in~\cite{AnIyMa24}, we omit explicit proofs. 
    \item Section~\ref{sec:main-results-proofs} is dedicated to proofs: the proof of Theorem~\ref{spat_flory} appearing in Section~\ref{sec:main}, and the proof of Theorem~\ref{thm:uniqueness} appearing in Section~\ref{sec:unique}.
\end{enumerate}

\subsection{The model, terminology and global assumptions} \label{intro-subsec-model}
\subsubsection{Definition of the cluster coagulation process}
 Recall that in the cluster coagulation process~\cite{norris-cluster-coag}, one begins with a configuration of \emph{clusters} in a measurable space $(E, \mathcal{B})$. Associated with a cluster $x \in E$ is a \emph{mass function} $m: E \rightarrow (0, \infty)$. Another important quantity associated with the process is a \emph{coagulation kernel} $K: E \times E \times \mathcal{B} \rightarrow [0, \infty)$, which satisfies the following:
\begin{enumerate}
    \item For all $A \in \mathcal{B}$ $(x,y) \mapsto K(x,y, A)$ is measurable,
    \item For all $x, y \in E$ $K(x,y, \cdot)$ is a measure on $E$,
    \item \emph{symmetric:} for all $A \in \mathcal{B}, x,y \in E$ $K(x,y, A) = K(y,x, A)$, 
    \item \emph{finite:} for all $x, y  \in E$ $\bar{K}(x,y) := K(x,y, E) < \infty$
    \item \emph{preserves mass:} for all $x, y \in E$, 
    $m(z) = m(x) + m(y)$ for $K(x,y, \cdot)$-a.a. $z \in E$.
\end{enumerate}

As a continuous time Markov chain, we associate the following dynamics. Suppose that we begin with a configuration of clusters. Then, 
\begin{itemize}
    \item to each pair of clusters $x, y \in E$, we associate an exponential random variable with parameter $\bar{K}(x,y)$;
    \item upon the elapsure of the next exponential random variable in the process, corresponding to the pair $x$ and $y$, say, the clusters $x$ and $y$ are removed and replaced by a new cluster $z \in E$, sampled according to the probability measure 
    \begin{equation}\label{measz}
    \frac{K(x, y, \cdot)}{\bar{K}(x, y)}.
    \end{equation}
\end{itemize}

\subsubsection{Notation and preliminaries} 

In this paper, we consider the process as depending on a parameter $N \in \mathbb{N}$, which one may consider as (up to random fluctuations) the total \emph{initial mass} of the system, and analyse the process as in the limiting regime as $N \to \infty$. 
We consider the configuration of clusters at time $t$ as being encoded by a \emph{random point measure} $\mathbf{L}^{(N)}_{t}$ on $E$, so that, for any set $A \subseteq E$, $a \in (0, \infty)$, $\mathbf{L}^{(N)}_{t}(A \cap m^{-1}([a, \infty)))$
denotes the random number of clusters of mass at least $a$ belonging to $A$. We denote by $\mathcal{M}_{+}(E)$ the set of finite, positive measures on $E$ (also defining $\mathcal{M}_{+}(E\times E)$ in a similar manner).  
We may then consider the process as a measure-valued Markov process, whose infinitesimal generator $\mathcal{A}$ is defined as follows: for any bounded measurable test function $F:\ME \rightarrow \mathbb{R}$, we have 
\begin{equation} \label{eq:gen-def}
\mathcal{A}F(\mathbf{L}_{t})=  \frac{1}{2}\int_{E\times E\times E} {\mathbf{L}_{t}}(\dd x) \left({\mathbf{L}_{t}} - \delta_{x}\right)(\dd y) K(x,y,\dd z) \left(F({\mathbf{L}}^{(x,y)\to z}_{t}) - F({\mathbf{L}_{t}})\right),
\end{equation}
where ${\mathbf{L}}^{(x,y)\to z} := {\mathbf{L}} + \left(\delta_{z} -\delta_{x} - \delta_{y} \right)$. Note that, as $\mathbf{L}^{(N)}_{t}$ is assumed to be a point measure, the above integral is always with respect to a positive measure. The measure ${\mathbf{L}}^{(x,y)}$ describes  the configuration of the system after a coagulation involving clusters the two clusters $x,y\in E$ and ending with one cluster $z$ with $m(z)= m(x) +m(y)$, for $K(x,y, \cdot)$-a.a. $z \in E$. Note the factor $\frac{1}{2}$ in front of the generator, present to ensure that the total rate at with clusters $x$ and $y$ interact is $\bar{K}(x,y)$ (and not $2\bar{K}(x,y)$).

In this paper, we assume $E$ is a $\sigma$-compact metric space 
with metric $d$. Given another space $F$, denote by $C_{b}(E; F)$, or resp. $C_{c}(E; F)$, the spaces of continuous functions $E \rightarrow F$ which are bounded, or resp., have compact support. In general, we write $C_{b}(E)$ (resp. $C_{c}(E)$) as a shorthand for $C_{b}(E;\mathbb{R})$ (resp. $C_{c}(E;\mathbb{R})$). We equip $\mathcal{M}_{+}(E)$ with a metric $d$ that induces the vague topology on $\mathcal{M}_{+}(E)$. \footnote{We recall that the \emph{vague} (respectively \emph{weak}) topologies on $\mathcal{M}_{+}(E)$ are the smallest topologies that make the maps $\mu \rightarrow \int_{E} f(x) \mu(\dd x)$ continuous for all $f \in C_{c}(E)$ (respectively $C_{b}(E)$). Note that, since $E$ is separable and complete, $(\mathcal{M}_{+}(E), d)$ is a separable and complete space.}
For any $n\in\N$, we denote by $\mathcal{E}_n$ the space
\[
\mathcal{E}_n := \left\{\mathbf{u} \in \mathcal{M}_{+}(E): \int_{E}m(x)\, \mathbf{u}(\dd x) \leq n \right\}
\]
and $\mathcal{E}=\bigcup_{n\in\N}\mathcal{E}_n$. Note that on $\mathcal{E}_n$ \emph{weak} and \emph{vague} topology coincides, while this is not the case on $\mathcal{E}$. We generally will regard $\mathcal{E}$ as the state space of the process, taking values in $D([0,\infty);\mathcal{E})$, the \emph{Skorokhod space} of right-continuous functions $f: [0,\infty) \rightarrow \mathcal{E}$ with \emph{Skorokhod metric} $d_{S}$ induced by $d$.

Given a measure $\mu \in \mathcal{M}_{+}(E)$ and a measurable function $f: E \rightarrow \mathbb{R}$, we denote by 
\[
\|\mu\|:=\int_{E}\mu(\dd x) \quad \text{and} \quad 
\langle f, \mu \rangle := \int_{E} f(x) \mu(\dd x).\] 
At the level of the stochastic process, we denote by $\Prob[N]{\cdot}$ and $\E[N]{\cdot}$ probability distributions and expectations with regards to the trajectories of the process with generator $\mathcal{A}$ and (possibly random) initial condition $\bar{\mathbf{L}}^{(N)}_{0}$. 

\subsubsection{Normalisation, global assumptions and the multi-type Flory equation}
We expect, \emph{a priori} that limiting equation, called the \emph{multi-type Flory equation} (extending the Smoluchowski equation from~\cite{norris-cluster-coag}) encodes the behaviour of the process $(\mathbf{L}^{(N)}_{t/N}/N)_{t \geq 0}$, for $N$ `large'. The re-scaling of time is required to counter-balance the increase in the number of interactions as the initial mass of clusters grows with $N$. Thus, in general, we set
\[
\bar{\mathbf{L}}^{(N)}_{t} := \mathbf{L}^{(N)}_{t/N}/N, 
 \]
and generally (by abuse of notation, since such a limit may not be unique) use $(\bar{\mathbf{L}}^{*}_{t})_{t \geq 0}$ to denote a weak limit (a limit along a subsequence) of the process.  

 Additionally, the initial condition of the process must meet certain natural convergence assumptions. Here, we outline the primary assumptions required. 
 \begin{assumption} \label{ass:global}
    In this paper, we assume throughout that 
\begin{enumerate}
    \item \label{item:ass-global-1} $E$ is a $\sigma$-compact metric space,
    \item \label{item:ass-global-2} we have $\bar{\mathbf{L}}^{(N)}_{0} = \sum_{i \in I} \frac{c_{i}\delta_{i}}{N}$ for some finite set $I\subseteq E$, $c_{i} \in \mathbb{N}$, and there exists $c' > 0$ such that $\sum_{i \in I}\frac{c_{i}}{N} \leq c'$ almost surely,  
    \item \label{item:ass-global-3} there exists $\mu \in \mathcal{E}$ such that 
\begin{equation} \label{eq:limit-init-cond}
\bar{\mathbf{L}}^{(N)}_{0} \rightarrow \mu
\end{equation}
weakly, in probability, 
and $\langle m, \mu\rangle>0$. 
\end{enumerate}
\end{assumption}

An informal definition of the multi-type Flory equation is as follows (see Definitions~\ref{weak_sol} and~\ref{cons_q} on the next page for a more formal treatment). Included in this equation, is a function $\phi$, which one may regard as a \emph{conserved quantity} of the system. Then,  $(\mathbf{u}_t, )_{t\geq 0}$, taking values in  $\mathcal{M}_{+}(E)$ is a solution of the \emph{multi-type Flory equation} with \emph{conserved quantity} $\phi$, and initial condition $\mathbf{u}_{0}$ if: 
\begin{equation}\label{flory}
\mathbf{u}_t - \mathbf{u}_0 = \int_{0}^{t} \left[Q^{+}(\mathbf{u}_s) - Q^{-}(\mathbf{u}_s)\right] \dd s;
\end{equation}
where $Q^{+}(\mathbf{u}_s)$ and $Q^{-}(\mathbf{u}_s)$ are measures defined such that, for appropriate test functions ${J \in C_{c}(E; \mathbb{R})}$, 
\begin{linenomath}
\begin{align} \label{eq:Q+}
&\int_{E} J(y) Q^{+}(\mathbf{u}_s)(\dd y) := \frac{1}{2}\int_{E\times E \times E} J(z) K(x,y, \dd z) \mathbf{u}_s(\dd x ) \mathbf{u}_s(\dd y),
\end{align}
\end{linenomath}
and 
\begin{linenomath}
\begin{align} \label{eq:Q-}
&\int_{E} J(y) Q^{-}(\mathbf{u}_s)(\dd x) :=  \int_{E\times E} J(y) \bar{K}(x,y) \mathbf{u}_s(\dd x ) \mathbf{u}_s(\dd y )   + \int_{E} J(y) g_{\infty}(y) \mathbf{u}_s(\dd y);
\end{align}
\end{linenomath}
with $g_{\infty}$ defined such that 
\begin{equation}
g_{\infty}(y) := \int_{E} \phi(x,y) \mathbf{u}_0(\dd x)  - \int_{E} \phi(x, y) \mathbf{u}_s(\dd x).
\end{equation}

\section{Statements of main results}
\subsection{Concentration of trajectories on solutions of multi-type Flory equations}
\label{sec:conc-trajectories}

We adapt the definition from \cite{norris-coag-99, norris-cluster-coag} of \emph{solutions} for the generalised Flory equation to this setting. 
\begin{defn}\label{weak_sol} 
Given a function $\phi: E\times E \rightarrow \mathbb{R}$, we say a map $t\mapsto \mathbf{u}_t \in \ME,$ is a \emph{solution of the multi-type Flory equation} with \emph{conserved quantity} $\phi$ if the following are satisfied:
\begin{enumerate}
    \label{item:def-1-1} \item for all Borel sets $A\subseteq E$ the map
    $t\mapsto \mathbf{u}_t(A)\colon [0,\infty)\to [0,\infty]$
    is measurable;
    \label{item:def-1-2} \item for all $f\in C_c(E)$, and $t \geq 0$, we have $\langle f,\mathbf{u}_0\rangle<\infty$,
    \begin{linenomath}
    \begin{align} \label{eq:flory-part-two}
    &\int_0^t\int_{E\times E} f(y) \bar{K}(x,y)\mathbf{u}_s(\dd x)\mathbf{u}_s(\dd y)\dd s<\infty; \\ & \hspace{3cm} \text{and} \quad\int_0^t\int_{E\times E} f(y) \phi(x,y)\mathbf{u}_0(\dd x)\mathbf{u}_s(\dd y)\dd s<\infty;
    \end{align}
    \end{linenomath}
    \label{item:def-1-3} \item for all $f\in C_c(E)$ and $t\geq 0$, with $Q^{+}$ and $Q^{-}$ as defined in~\eqref{eq:Q+} and~\eqref{eq:Q-}, 
    \begin{equation} \label{eq:flory-test-function}
    \langle f, \mathbf{u}_t\rangle=\langle f, \mathbf{u}_0\rangle+\int_0^t\langle f, Q^{+}(\mathbf{u}_s) - Q^{-}(\mathbf{u}_s)\rangle \, \dd s.
    \end{equation}
    \label{item:def-1-4} \item For each $x \in E$, $t \geq 0$ we have 
    \begin{equation} \label{eq:annoying-cons-assumption}
        \int_{E} \phi(x,u) \mu_{t}(\dd u) \leq \int_{E} \phi(x,u) \mu_{0}(\dd u).
    \end{equation}
\end{enumerate}
Note that Item~\ref{item:def-1-2} ensures that the equation in Item~\ref{item:def-1-3} is well-defined (without terms of the form $\infty - \infty$).
\end{defn}

As alluded in the Definition~\ref{weak_sol}, an important feature of a multi-type Flory equation is a \emph{conserved quantity} $\phi$. This as a function $\phi: E \times E \rightarrow \mathbb{R}$ such that, for any $x$, the quantity $\langle \phi(\cdot, x), \bar{\mathbf{L}}^{(N)}_{t}\rangle$ is fixed for each $t > 0$. Perhaps the most natural conserved quantity in a coagulation process is \emph{mass}, which corresponds to the choice of function $\phi(x,y) = m(x)$; as masses add upon coagulation, this is fixed along trajectories of the process. However, one may imagine, in models encoding more information about clusters, that there are other quantities conserved, reflecting, for example, the \emph{centre of mass} of clusters in space, or the \emph{momentum} of the system. 
\begin{defn}[Conserved or sub-conserved quantities] \label{cons_q}
    A function $\phi: E \times E \rightarrow \mathbb{R}$ is said to be \emph{conservative} or a \emph{conserved quantity} if for all $x,y,q \in E$, for $K(x,y,\cdot)$ a.a. $z$,  
\begin{equation} \label{eq:good-kernels}
\phi\left(z, q\right) = \phi(x, q) + \phi(y,q).
\end{equation}
    It is, similarly, said to be \emph{sub-conservative} if for all $x,y,q \in E$, for $K(x,y,\cdot)$ a.a. $z$,  
\begin{equation} \label{eq:sub-cons-kernels}
\phi\left(z, q\right) \leq \phi(x, q) + \phi(y,q).
\end{equation}
It is said to be \emph{doubly conservative} (similarly \emph{doubly sub-conservative}), if it is conservative in the second argument in addition to the first, so that in addition for all $x,y,q \in E$, for $K(x,y,\cdot)$ a.a. $z$,
\begin{equation} \label{eq:good-kernels-2}
\phi\left(q, z\right) = \phi(q, x) + \phi(q, y).
\end{equation}
Finally, if a function $\xi: E \rightarrow \mathbb{R}$ is such that $\phi(x,y) = \xi(x)$ is conservative (resp. sub-conservative), we also say $\xi$ is conservative (resp. sub-conservative). 
\end{defn}

\begin{examples}
Two natural examples of conserved quantities are $\phi \equiv 0$, and $\phi(x, y) = m(x) \ell(y)$ for some measurable function $\ell: E \rightarrow \mathbb{R}_{+}$. In the prior case, the associated multi-type Flory equation corresponds to the generalised Smoluchowski equation introduced in~\cite[Section~2]{norris-cluster-coag}, whilst the latter case corresponds to the Flory equation (called the `modified Smoluchowski equation' in~\cite{norris-cluster-coag}). 
\end{examples}

\begin{examples} \label{rem:doubl-cons}
    If $\phi$ is symmetric and conservative, it is also doubly conservative. Thus, some examples of doubly conservative functions include $\phi(x,y) = m(x)m(y)$. Another example comes from the \emph{bilinear coagulation process}, studied in~\cite{heydecker2019bilinear}, where  $E = [0, \infty)^{d}$, $A \in [0, \infty)^{d \times d}$ is a symmetric matrix with non-negative entries and $K(x, y, \dd z) = (x^{T} A y) \delta_{x+ y}$. In this case, for any matrix $M$ the function $\phi(x,y)=x^{T} M y$ is a doubly conservative function. This means that $\bar{K} = x^{T} A y$ is itself is a doubly conservative function. 
\end{examples}

\begin{rmq}
        Note that, if $\xi':[0, \infty) \rightarrow [0, \infty)$ is continuous and sub-additive, the function $\phi'(x,y) = (\xi'(m(x)))(\xi'(m(y))) $ is doubly sub-conservative. This is the analogue of `sublinear' function used by Norris in~\cite{norris-cluster-coag}.
\end{rmq}

Our main assumptions are as follows. The first, Assumption~\ref{ass:tightness} ensures that the probability measures $(\mathbb{P}_{N})_{N \in \mathbb{N}}$ on the space $D([0, \infty); \mathcal{E})$ induced by the processes $(\bar{\mathbf{L}}^{(N)}_t)_{t \in [0, \infty)}$ are tight; and thus, by Prokhorov's theorem, the collection of random trajectories $\left\{(\bar{\mathbf{L}}^{(N)}_t)_{t \in [0, \infty)},  N \in \mathbb{N}\right\}$  contains weakly convergent subsequences. 

\begin{assumption}[Conditions for tightness] \label{ass:tightness}
Assume the following:
    \begin{enumerate}
        \item \label{item:ass-doubl-cons} There exists a doubly sub-conservative $\phi'$ such that $\bar{K} \leq \phi'$ pointwise.
\item \label{item:ass-limsup-cond} We have 
\begin{linenomath}
\begin{align} \label{eq:conv-cons-prob}
\limsup_{N \to \infty} \E[N]{\int_{E\times E} \bar{\mathbf{L}}^{(N)}_{0} (\dd x )\left(\bar{\mathbf{L}}^{(N)}_{0} (\dd y) -\frac{\delta_x}N\right)\phi'(x,y)} < \infty.
\end{align}
\end{linenomath} 
\item \label{item:ass-doubl-cons-compact} There exists a doubly sub-conservative function $\phi'': E \times E \rightarrow [0, \infty)$, such that for all $n \in \mathbb{N}$ the set
\begin{equation} \label{eq:bounded-projection}
\mathcal{E}^{*}_{n} := \left\{\mathbf{u} \in \mathcal{M}_{+}(E \times E) : \int_{E} \mathbf{u}(\dd x \times \dd y ) \phi''(x, y) \leq n \right\}
\end{equation}
is compact, and $\phi''$ satisfies~\eqref{eq:conv-cons-prob}.
    \end{enumerate}
\end{assumption}
Our next assumption, Assumption~\ref{ass:conv-to-flory} (which contains Assumption~\ref{ass:tightness}), are criteria under which limits of sub-sequences of the collection $\left\{(\bar{\mathbf{L}}^{(N)}_t)_{t \in [0, \infty)},  N \in \mathbb{N}\right\}$ are concentrated on trajectories that solve the multi-type Flory equation. In particular, Item~\ref{item:bounded_compact} from Assumption~\ref{ass:conv-to-flory} intuitively states that the coagulation kernel $\bar{K}$ is more and more comparable to a conservative function $\phi$ outside larger and larger compact sets. 
\begin{assumption}[Conditions ensuring concentration of trajectories] \label{ass:conv-to-flory}
Assume that the following hold.
\begin{enumerate}
 \item \label{item:continuity} Assumption~\ref{ass:tightness} is satisfied and the functions $\bar{K}$ and $\phi'$ are continuous.
   \item \label{item:bounded_compact} There exists a continuous, conservative function $\phi$ satisfying~\eqref{eq:conv-cons-prob}, such that one of the following hold: 
   \begin{enumerate}
   \item  For an increasing collection of sets $(C_{k})_{k \in \mathbb{N}} \subseteq E$, with $\bigcap_{k \in \mathbb{N}} \overline{C^{c}_{k}} = \varnothing$ we have, for any compact $C' \subseteq E$
    \begin{linenomath}
        \begin{align} \label{bounded-outside-compact}
        & \limsup_{k \to \infty} \sup_{x \in C^{c}_{k}, \, y \in C'}  \left|\bar{K}(x,y)- \phi(x,y) \right| < \infty.
        \end{align}
    \end{linenomath}
    \item     There exists a continuous doubly sub-conservative function $\phi^{*}$ satisfying Equation~\eqref{eq:conv-cons-prob}, such that, for a collection of sets $(C_{k})_{k \in \mathbb{N}} \subseteq E$, we have, for any compact $C' \subseteq E$
    \begin{linenomath}
        \begin{align} \label{eq:lim-cons-kern}
        & \lim_{k \to \infty} \sup_{x \in C^{c}_{k}, \, y \in C'} \frac{\left|\bar{K}(x,y)- \phi(x,y) \right|}{\phi^{*}(x,y)} = 0.
        \end{align}
    \end{linenomath}
\end{enumerate}
   \item  \label{item:init_cond} The limiting initial condition $\mu$ from~\eqref{eq:limit-init-cond} is such that for any compact set $C' \subseteq E$ 
\begin{linenomath}
\begin{align} \label{eq:conv-probab} 
& \lim_{N \to \infty} \sup_{y \in C'} \bigg| \int_{E} \bar{\mathbf{L}}^{(N)}_{0}(\dd x) \phi(x,y) - \int_{E} \mu(\dd x)  \phi(x,y) \bigg| = 0 \quad \text{almost surely.}
\end{align}
\end{linenomath}
\end{enumerate}
\end{assumption}
The main result of this section is the following.
\begin{thm}\label{spat_flory}
Let $\left\{(\bar{\mathbf{L}}^{(N)}_t)_{t \in [0, \infty)},  N \in \mathbb{N}\right\}$ be a sequence of cluster coagulation processes. Then, the following holds:
\begin{enumerate}
\item \label{item:tightness} If Assumption~\ref{ass:tightness} is satisfied, the sequence $\left\{(\bar{\mathbf{L}}^{(N)}_t)_{t \in [0, \infty)},  N \in \mathbb{N}\right\}$ contains converging sub-sequences;
\item  \label{item:conv-to-flory}  If Assumption~\ref{ass:conv-to-flory} is satisfied, the limit $\bar{\mathbf{L}}^{*}$ of any weakly convergent subsequence is, almost surely, a solution of the multi-type Flory equation with conserved quantity $\phi$ and initial condition $\mu$ (according to Definition \ref{weak_sol}).\\
\end{enumerate} 
\end{thm}

\begin{rmq} \label{rem:norris-uniqueness}
 If we have uniqueness of the solution of the associated Flory equation, then Theorem~\ref{spat_flory} also implies a weak law of large numbers for the cluster coagulation process, (since weak convergence to a constant implies convergence in probability). This has been shown by Norris~\cite{norris-cluster-coag} in the cases that either $\bar{K}(x,y) \leq m(x) + m(y)$, $\bar{K}(x,y) = \xi'(m(x))\xi'(m(y))$ where $\xi':[0, \infty) \rightarrow [0, \infty)$ is continuous, sub-linear and $(\xi')^2$ is sub-linear, or if the kernel is `eventually multiplicative' (meaning that, for some $R > 0$, $\bar{K}(x,y) = m(x) m(y)$ on $(m^{-1}([0,R]) \times m^{-1}([0,R]))^{c}$; see~\cite[Theorem~2.3]{norris-cluster-coag}).
    Notably, Norris also shows an exponential rate of convergence of the coagulation process to the limit, in a the restricted case of `polymer models' when the limit is unique (see~\cite[Theorem~4.2 and Theorem~4.3]{norris-cluster-coag}). 
We note, however, that the setting we consider here is more general, and there are a wide range of kernels satisfying Equation~\eqref{eq:lim-cons-kern} which are not eventually multiplicative. The natural extension of this result is to 
consider kernels that are, in some sense, ``eventually conservative'' (c.f. Theorem~\ref{thm:uniqueness}).
\end{rmq}

\begin{example}
Some particular instances where Theorem~\ref{spat_flory} applies, are listed in the following.
\begin{enumerate}
    \item We can take $\phi' = \phi'' = \phi^{*}$ in Theorem~\ref{spat_flory} and Lemma~\ref{lem:tight}, with $\phi'$ satisfying~\eqref{eq:conv-cons-prob}. In this case, if there exist a nested sequence of compact sets $(C_{k})_{k \in \mathbb{N}}$ such that $\bigcup_{k \in \mathbb{N}} C_{k} = E$ and $\phi'(x,y) > k$ on $(C_{k} \times C_{k})^{c}$, one readily verifies the compactness in~\eqref{eq:bounded-projection}, by Prokhorov's theorem and Markov's inequality. Often, these sets $(C_{k})_{k \in \mathbb{N}}$ can also be used in~\eqref{bounded-outside-compact} or~\eqref{eq:lim-cons-kern}.
    \begin{enumerate}
    \item For example, assuming continuity of $\bar{K}$ we can take 
    \begin{equation} \label{eq:ex-sublinear-mass}
    \phi'(x,y) = \xi'(m(x)) \xi'(m(y))
    \end{equation}
    for some continuous, sub-linear, positive function $\xi'$, with $\lim_{k \to \infty} \xi'(k) = \infty$, and, assuming $m^{-1}([0,k])$ is compact, choose $C_{k} = m^{-1}([0,k])$. This is a common assumption in the setting of hydrodynamic limits converging to the Smoluchowski equation for example in~\cite[Equation~(2.4)]{norris-cluster-coag}. 
    \item In a similar setting, we can choose $\phi'(x,y) = m(x) \wedge m(y)$, in which case~\eqref{bounded-outside-compact} is satisfied with $\phi \equiv 0$. It may the case that 
    \[
    \int_{E \times E} m(x) \wedge m(y) \mu(\dd x) \mu(\dd y) < \infty, 
    \]
    but $\pairing{m, \mu} = \infty$; and, as far as we are aware, this case is not covered by previous results appearing in the literature. This offers an alternative to the assumption~\cite[Equation~(2.4)]{norris-cluster-coag} in proving convergence to solutions of the Smoluchowski equation. In a similar vein, we can choose $\phi'(x,y) = \xi'(m(x)\wedge m(y))$, for a continuous, sub-linear, positive function $\xi'$. 
    \item In the setting of the \emph{bilinear coagulation process} (see Examples~\ref{rem:doubl-cons}), we can take $\phi = \phi' = \bar{K}(x,y) = x^{T}A y$. If, for example, the matrix $A$ has non-zero entries, the sets $\{x, y \in [0, \infty)^{d}: x^{T} A y \leq k\}$ are compact, and so Equation~\eqref{eq:bounded-projection} is satisfied. 
    \end{enumerate}
    \item Alternatively, we may choose $\phi(x,y)=m(x)\ell(y)$ for some measurable function $\ell:E\to \R_+$, $\phi^*(x,y)=m(x)$, and $\phi'(x,y) = \xi'(m(x)) \xi'(m(y))$, and $\lim_{N \to \infty} \massminus{\bar{\mathbf{L}}^{(N)}_{0}} = \massminus{\mu}$ almost surely (which may be used to show~\eqref{eq:conv-probab}).  In the setting of the classical Marcus-Lushnikov process, this example includes
     a mild strengthening of~\cite[Theorem~2.3]{fournier-giet-04}, showing concentration of trajectories around the classical Flory equation. However, we allow for random initial conditions $(\bar{\mathbf{L}}^{(N)}_{0})_{N \in \mathbb{N}}$, and do not require $\xi'(x) \geq 1$.
\end{enumerate}
\end{example}



\begin{example}\label{ex2}
In~\cite[Example~3.3]{AnIyMa24}, a number of toy spatial coagulation models are introduced. Following the notation introduced there,
suppose that
\begin{equation}\label{lim}
\lim_{n\to\infty} \frac{\bar{K}((p,n), (s,o))}{n} =\ell(s,o) <\infty,
\end{equation}
where $\ell$ is a continuous function that does not depend on the position $p$ of the mass going to infinity. Then,
the function $\phi:= n\ell(s,o)$ satisfies the properties in \eqref{eq:lim-cons-kern} and \eqref{eq:good-kernels} for any choice of the location of the newly formed particle, which we indicate by $X((p, m), (s,n))$. Indeed \eqref{lim} implies condition \eqref{eq:lim-cons-kern}, and we can also check that also \eqref{eq:good-kernels} is satisfied: 
\begin{linenomath}
\begin{align*}
\phi\left(\left(X((p, m), (s,n)), m + n \right), (v, j)\right) &= (m+n)\ell(v,j) = \phi((p, m)), (v, j)) + \phi((s,n), (v,j)).
\end{align*}
Thus, in this context, we can apply Theorem~\ref{spat_flory} to deduce a spatial analogue of the convergence to the Flory equation in~~\cite[Theorem~2.3]{fournier-giet-04}.
\end{linenomath}
\end{example}

\begin{rmq}
        If one considers the initial condition $\bar{\mathbf{L}}^{(N)}_{0} := \frac{\sum_{i=1}^{N} \delta_{X_{i}}}{N}$, where $X_{i}$ are i.i.d samples from the limiting measure $\mu$, if we choose $\phi'$ according to~\eqref{eq:ex-sublinear-mass}, by applying the strong law of large numbers, one can readily verify that Equations~\eqref{eq:conv-cons-prob} is satisfied when $\langle \xi' \circ m, \mu \rangle < \infty$. 
In this case, the term $\delta_{x}/N$ appearing in Equation~\eqref{eq:conv-cons-prob} 
is crucial for this argument to work, since it may be the case, for example, that $\langle \xi' \circ m \rangle < \infty$, 
but $\langle (\xi' \circ m)^2 \rangle = \infty$. 
\end{rmq}

\subsection{Sufficient condition for uniqueness: eventually conservative kernels}
\label{sec:unique}
Uniqueness of solutions of a multi-type Flory equation is a rather delicate issue even for the particular case of the classical Smoluchowski equation (see~\cite{norris-coag-99} for an example of kernel which gives non-unique solutions). In the following theorem we focus on certain kernels $\bar{K}(x,y)$ which coincide with a conservative function outside a compact set. In analogy with~\cite[Theorem 2.3]{norris-cluster-coag} we call these kernels `eventually conservative' and we prove that this condition is sufficient to ensure uniqueness of solutions. 
\begin{thm} \label{thm:uniqueness}
    Suppose that $(\mu_{t})_{t \geq 0}$ is a solution to the multi-type Flory equation (according to Definition \ref{weak_sol}) with conserved quantity $\phi$ and initial condition $\mu_{0}$ such that $\|\mu_0\| < \infty$. Also, assume that, for each $k \in \mathbb{N}$ sufficiently large, the set 
\begin{equation}\label{D}
D_{k} := \left\{x \in E: \int_{E} \phi(x, y) \mu_{0}(\dd y) \leq k \right\}
\end{equation}
is compact, with $\bigcup_{k \in \mathbb{N}} D_{k} = E$;  we have $\bar{K}(x,y) \leq c' \phi(x,y)$ for some $c' > 0$ and for some $R$ sufficiently large,
$\phi(x,y) = \bar{K}(x,y)$ on $(D_{R} \times D_{R})^{c}$. Then, the solution $(\mu_{t})_{t \geq 0}$ is unique. 
\end{thm}
\begin{cor}\label{cor:LLN}
    Under the hypotheses of  Theorem~\ref{spat_flory} and Theorem~\ref{thm:uniqueness}; if  $(\mu_{t})_{t \geq 0}$ denotes the associated unique solution to the multi-type Flory equation, we have 
    \[
    (\bar{\mathbf{L}}^{(N)})_{t \geq 0} \longrightarrow  (\mu_{t})_{t \geq 0} \quad \text{ in probability, in  $D([0, \infty); \mathcal{E})$.}
    \]
\end{cor}
\begin{rmq}
Note that, unlike in the context of Theorem~\ref{spat_flory}, the condition that $\phi(x,y) = \bar{K}(x,y)$ on $(D_{R} \times D_{R})^{c}$, combined with the symmetry of $\bar{K}$ implies that $\phi(x,y)$ is symmetric. 
\end{rmq}

\begin{rmq}
    If we take $\phi(x,y) = m(x) m(y)$, one readily verifies that $\bar{K}(x,y)$ is eventually conservative in the sense of Theorem~\ref{thm:uniqueness} if the sets $m^{-1}([0, R])$  are compact. 
\end{rmq}

\begin{rmq}
    Note that, in the setting of the \emph{bilinear coagulation process} (see Remark~\ref{rem:doubl-cons}), the kernel $\bar{K}(x,y) = x^{T}A y$ is a symmetric conservative function, hence the uniqueness result, and weak law of large numbers from Theorem~\ref{thm:uniqueness} and Corollary~\ref{cor:LLN} extends the results of~\cite{heydecker2019bilinear}.
\end{rmq}

\begin{rmq}
 It is possible to also apply these results to more general spaces $E$; as long as $E$ is a $\sigma$-compact metric space. This means that we can take $E$ to be, for example, an appropriate (restricted) space of functions  and the kernel to be of the form \[K(f,g,\dd h)=\delta_{\frac{m(f)f+m(g)g}{m(f) + m(g)}}\bar{K}(f,g).\] With any kernel of this form, a symmetric, bilinear form gives rise to a symmetric, conservative function $\phi$; for example, if $E = C([0,1]; [0,1])$ we see that $\phi(f,g)=\int f(x)g(x)\dd x$ is symmetric and conservative.
\end{rmq}

\subsection{Existence of gelling solutions to the multi-type Flory equation}\label{sec:existence_gelling}
In the companion paper~\cite{AnIyMa24}, we focus on gelation for the cluster coagulation model. In particular, we derive a general, sufficient criteria for \emph{stochastic gelation} (see~\cite{AnIyMa24} for the different definitions of gelation) of the sequence of cluster coagulation processes $\left\{(\bar{\mathbf{L}}^{(N)}_t)_{t \in [0, \infty)},  N \in \mathbb{N}\right\}$.
The following is an immediate implication of Theorem~\ref{spat_flory} combined with~\cite[Theorem~1.1 and Theorem~2.3]{AnIyMa24}, hence we omit providing an explicit proof.
\begin{cor} \label{cor:class-gel-sol} Suppose a cluster coagulation process $(\bar{\mathbf{L}}^{(N)}_{t})_{t\geq 0}$ satisfies {Assumptions~2.1 \& 2.2} from~\cite{AnIyMa24} and the conditions of Theorem~\ref{spat_flory}, and that $m$ is continuous. Then, there exists a gelling solution to the multi-type Flory equation (Definition~\ref{weak_sol}), with conserved quantity $\phi$ and initial condition $\mu$.
\end{cor}

\begin{rmq}
    Note that it may be the case that {Assumptions~2.1 \& 2.2} in \cite{AnIyMa24} are satisfied, but $\massminus{\mu} = \infty$. In this case, if the conditions of Theorem~\ref{spat_flory} apply,
trajectories of the cluster coagulation process still concentrate around solutions of a Smoluchowski equation; and thus it is natural to use the notion of \emph{stochastic gelation} to define gelling solutions of such equations.  
\end{rmq}
     One of the novelties of Corollary~\ref{cor:class-gel-sol} 
     comes from the criterion for gelation of classical coagulation processes under the conditions of \cite[Corollary~2.4]{AnIyMa24}. In particular, this confirms that a large class of homogeneous kernels, with exponent $\gamma > 1$ have gelling solutions, a well-known conjecture from scientific modelling literature~\cite{aldous99, eibeck-wagner-01}. Previously, Wagner showed that the \emph{mass flow process} associated with such models is \emph{explosive}, a property conjectured to hold for all coagulation kernels with gelling solutions~\cite{eibeck-wagner-01, wagner-explosive-phenomena}. Since we believe this is an interesting result on its own, we state it separately in the following corollary.
     \begin{cor}  \label{cor:gel-2} Suppose we are in the setting of the classical Marcus-Lushnikov process, that is $E = (0, \infty)$, $K(x, y, \dd z) =\bar{K}(x,y) \delta_{x + y}$, for a continuous symmetric function $\bar{K}(x,y)$, and the mass function $m(x) \equiv x$. Let $\mass{\mu}>0$ and the conditions of Theorem~\ref{spat_flory} hold. Suppose one of the following conditions hold:
    \begin{enumerate}
        \item[1.] we have $\inf_{i\in [1, 2]} \bar{K}(1, i) > 0$ and for all $x, y$ sufficiently large \[\bar{K}(cx, cy) = c^{\gamma} \bar{K}(x, y), \text{ with } \gamma > 1\]
    \end{enumerate}
    or, alternatively, 
\begin{enumerate}
        \item[2.] there exists $\epsilon > 0$ such that, for all $x, y$ sufficiently large \[\bar{K}(x, y) \geq (x \wedge y) \log{(x \wedge y)}^{3 + \epsilon}.\]
    \end{enumerate}
    Then, there exists a gelling solution to the multi-type Flory equation (Definition~\ref{weak_sol}), with conserved quantity $\phi$ and initial condition $\mu$.
\end{cor}

\section{Proofs of main results} \label{sec:main-results-proofs}
In Section~\ref{sec:main} we prove the main result of Theorem~\ref{spat_flory}, whilst in Section~\ref{sec:unique} we prove Theorem~\ref{thm:uniqueness}. Note that the proof of Theorem~\ref{thm:uniqueness} requires considerably less technicalities, and may be read independently of Section~\ref{sec:main}. 

\subsection{Proof of Theorem~\ref{spat_flory}}
\label{sec:main}
The normalised cluster coagulation process $(\bar{\mathbf{L}}^{(N)}_{t})_{t\geq 0}$ induces
a probability measure $\mathbb{P}_{N}$ on the space $D([0,\infty);\mathcal{E})$. To prove Item~\ref{item:tightness} of Theorem~\ref{spat_flory}, in Section~\ref{sec:tightness}, we show that under Assumption~\ref{ass:tightness} the family of probability measures $(\mathbb{P}_{N})_{N \in \mathbb{N}}$ is tight, which implies by Prokhorov's theorem~\cite[Theorem~IV.29, page 82]{pollard}
that $(\mathbb{P}_{N})_{N \in \mathbb{N}}$ has a weakly convergent subsequence. We then show in Section~\ref{sec:accumulation}, that under Assumption~\ref{ass:conv-to-flory} any such subsequence concentrates on solutions of the \emph{multi-type Flory equation}. 

Throughout this section, it will be beneficial to have associated ``conserved'' quantities, preserved by the dynamics of the process. The following lemma will be useful throughout:

\begin{lemma}\label{lem:conservation}
    For any cluster coagulation process $(\bar{\mathbf{L}}^{(N)}_{t})_{t\geq 0}$, given any doubly sub-conservative $\phi^{'}: E \times E \rightarrow \mathbb{R}$, almost surely for all $t \geq 0$ we have, for each $y \in E$
    \begin{equation} \label{eq:conservation-1}
        \int_{E} \left(\bar{\mathbf{L}}^{(N)}_{t}(\dd x)\right) \phi'(x,y) \leq \int_{E} \left(\bar{\mathbf{L}}^{(N)}_{0}(\dd x)\right) \phi'(x,y)
    \end{equation}
and 
\begin{linenomath}
    \begin{align} \label{eq:conservation-2}
        \int_{E \times E} \bar{\mathbf{L}}^{(N)}_{t}(\dd x) \left(\bar{\mathbf{L}}^{(N)}_{t}(\dd y) - \frac{\delta_{x}}{N}\right) &\phi'(x,y) \leq \int_{E \times E} \bar{\mathbf{L}}^{(N)}_{0}(\dd x) \left(\bar{\mathbf{L}}^{(N)}_{0}(\dd y) - \frac{\delta_{x}}{N}\right) \phi'(x,y). 
    \end{align}
   \end{linenomath} 
\end{lemma}

\begin{proof}
For Equation~\eqref{eq:conservation-1} if $\tau_{1} < \tau_{2}$ denote times of two consecutive coagulation events, with $\tau_{2}$ involving the coagulation of clusters $x'$ and $y'$ to a new cluster $z$, we have 
\begin{linenomath}
    \begin{align*}
        \int_{E}\left(\bar{\mathbf{L}}^{(N)}_{\tau_2}(\dd x)  - \bar{\mathbf{L}}^{(N)}_{\tau_1}(\dd x) \right) \phi'(x,y) = \left( \phi'(z,y) - \phi'(x', y) - \phi'(y', y)\right). 
    \end{align*}
\end{linenomath}
The right-hand side is $0$ for $K(x', y', \dd z)-$a.a. $z$, hence almost surely. 
Now, for Equation~\eqref{eq:conservation-2}, note that integrals of $\phi'(x,y)$ with respect to the product measure 
\[\bar{\mathbf{L}}^{(N)}_{s}(\dd x) \left(\bar{\mathbf{L}}^{(N)}_{s}(\dd y) - \frac{\delta_{x}}{N}\right)\] are nothing but sums of $\phi'(x,y)$ across all the distinct pairs of clusters $x, y$ in the process at time $s$. Thus, if $\tau_{1} < \tau_{2}$ denote times of two consecutive coagulation events, with $\tau_{2}$ involving the coagulation of clusters $x'$ and $y'$ to a new cluster $z$, 
we have
\begin{linenomath}
\begin{align*}
&\int_{E \times E} \bar{\mathbf{L}}^{(N)}_{\tau_2}(\dd x) \bigg(\bar{\mathbf{L}}^{(N)}_{\tau_2}(\dd y) 
 - \frac{\delta_{x}}{N}\bigg) \phi'(x, y) -\int_{E \times E} \bar{\mathbf{L}}^{(N)}_{\tau_1}(\dd x) \left(\bar{\mathbf{L}}^{(N)}_{\tau_1}(\dd y) - \frac{\delta_{x}}{N}\right) \phi'(x, y)
\\ & \hspace{1cm} = -\frac{\phi(x',y') + \phi(y',x')}{N^2} + \int_{E} \left(\bar{\mathbf{L}}^{(N)}_{\tau_1} - \frac{\delta_{x'}}{N} - \frac{\delta_{y'}}{N}\right)(\dd u) \frac{\left(\phi'(z,u)  - \phi'(x',u) -   \phi'(y',u) \right)}{N^2}
\\ & \hspace{2cm} + \int_{E} \left(\bar{\mathbf{L}}^{(N)}_{\tau_1} - \frac{\delta_{x'}}{N} - \frac{\delta_{y'}}{N}\right)(\dd u) \frac{\left(\phi'(u,z)  - \phi'(u,x') -   \phi'(u,y') \right)}{N^2} \leq 0,
\end{align*}
\end{linenomath}
for $K(x',y', \dd z)-$ a.a. $z$. Note that the first term in the second line comes from the contribution to the integral from the pair $x', y'$ involved with the coagulation, and the other integrals in the second and third line, comes from the difference in the contributions to the integrals from pairs $(v,u)$ where $v \in \{x', y', z\}$, and $u$ is a cluster not involved in the coagulation event. The result follows by iterating over the jumps in the process. 
\end{proof}
\subsubsection{Tightness: proof of Item~\ref{item:tightness} of Theorem~\ref{spat_flory}} \label{sec:tightness}
Item~\ref{item:tightness} of Theorem~\ref{spat_flory} is a consequence of tightness of the sequence of measures  $(\mathbb{P}_{N})_{N \in \mathbb{N}}$, which we prove in the following lemma. 
\begin{lemma}\label{lem:tight}
Assume that Assumption~\ref{ass:tightness} is satisfied. Then the sequence of probability measures $(\mathbb{P}_{N})_{N \in \mathbb{N}}$ is a tight family of probability measures on the Skorokhod space $D([0, \infty); \mathcal{E})$.
\end{lemma}
\begin{proof}[Proof of Item~\ref{item:tightness} of Theorem~\ref{spat_flory}]
This is an immediate consequence of Lemma~\ref{lem:tight} and Prokhorov's theorem~\cite[Theorem~IV.29, page 82]{pollard}.
\end{proof}
In order to prove Lemma~\ref{lem:tight}, we apply some well-established tightness criterion, stated in Appendix~\ref{app:tightness-criteria}. 
\begin{proof}[Proof of Lemma~\ref{lem:tight}]
We apply Lemma~\ref{lem:jakubowski}. For the first compact containment criterion, first recall that by~\eqref{eq:bounded-projection}, the set 
\[\mathcal{E}^{*}_{n} := \left\{\mathbf{u} \in \mathcal{M}_{+}(E \times E) : \int_{E} \mathbf{u}(\dd x \times \dd y ) \phi''(x, y) \leq n \right\}
 \]
is compact.  
Now, note that, by Lemma~\ref{lem:conservation}, if we have 
$$
\bar{\mathbf{L}}^{(N)}_{0}(\dd x) \left(\bar{\mathbf{L}}^{(N)}_{0}(\dd y) - \frac{\delta_{x}}{N}\right) \in \mathcal{E}^{*}_{n}
$$
then, for all $s \geq 0$, $\bar{\mathbf{L}}^{(N)}_{s}(\dd x) \left(\bar{\mathbf{L}}^{(N)}_{s}(\dd y) - \frac{\delta_{x}}{N}\right) \in \mathcal{E}^{*}_{n}.$\\
Suppose that we denote by $\mathcal{D}_{N}$ the set
\[
\mathcal{D}_{N} := \left\{\mathbf{u} \in \mathcal{E}: \mathbf{u} = \sum_{i \in I} \frac{c_{i} \delta_{i}}{N}, \: c_{i} \in \mathbb{N}, \: I \subseteq E\right\}, 
\]
and $\iota_{N}: \mathcal{D}_{N} \rightarrow \mathcal{M}_{+}(E \times E)$ denotes the map $\mathbf{u} \mapsto \mathbf{u}(\dd x) \left(\mathbf{u}(\dd y) - \frac{\delta_{x}}{N}\right)$; and extend this map to a map $\iota:\bigcup_{N \in \mathbb{N}} \mathcal{D}_{N} \rightarrow \mathcal{M}_{+}(E \times E)$ such that $\iota \equiv \iota_{N}$ on $\mathcal{D}_{N}$.
We now note that for any $n \in \mathbb{N}$ the set 
\[
B_{n} := \left\{\mathbf{u} \in \bigcup_{N \in \mathbb{N}} \mathcal{D}_{N}: \iota(\mathbf{u}) \in \mathcal{E}^{*}_{n} \right\}
\]
is relatively compact. Indeed, by the compactness of $\mathcal{E}^{*}_{n}$, any sequence $(\iota(\mathbf{u}_{i}))_{i \in \mathbb{N}}$ has a convergent subsequence $(\iota(\mathbf{u}_{i_{k}}))_{k \in \mathbb{N}}$. Suppose $\nu$ denotes a limit of this subsequence. There, are two cases: we can either find a further subsequence (which we also denote $(\iota(\mathbf{u}_{i_{k}}))_{k \in \mathbb{N}}$), such that,  for some $N' \in \mathbb{N}$ we have $(\iota(\mathbf{u}_{i_{k}}))_{k \in \mathbb{N}} = (\iota_{N'}(\mathbf{u}_{i_{k}}))_{k \in \mathbb{N}}$, or it is the case that for any $N' \in \mathbb{N}$ there exists $k \in \mathbb{N}$ such that $\iota(\mathbf{u}_{i_{k}}) = \iota_{j}(\mathbf{u}_{i_{k}})$ for some $j \geq N'$. In the latter case, (since the co-efficient of the $\delta_{x}$ term tends to $0$), we readily verify that 
\[
\mathbf{u}_{i_{k}} \otimes \mathbf{u}_{i_{k}} \rightarrow \nu, 
\]
hence $(\mathbf{u}_{i_{k}})_{k \in \mathbb{N}}$ also converges weakly. We may similarly deduce the result in the first case, when  
\[
\mathbf{u}_{i_{k}}(\dd x) \left(\mathbf{u}_{i_{k}}(\dd y) - \frac{\delta_{x}}{N'}\right) \rightarrow \nu(\dd x \times \dd y).
\]
Now, since $\phi''$ satisfies~\eqref{eq:conv-cons-prob} we know $\E[N]{\int_{E\times E} \bar{\mathbf{L}}^{(N)}_{0} (\dd x )\left(\bar{\mathbf{L}}^{(N)}_{0} (\dd y) -\frac{\delta_x}N\right)\phi''(x,y)} < c_0$, for some $c_0\in\mathbb{N}$. Therefore, by Markov's inequality, for any $c_1 \in \mathbb{N}$, 
\[
 \liminf_{N \to \infty} \Prob[N]{ \forall t \geq 0 \; \bar{\mathbf{L}}^{(N)}_{t} \in B_{c_{1}}}= \liminf_{N \to \infty} \Prob[N]{\left \langle \phi'', \iota(\bar{\mathbf{L}}^{(N)}_0) \right \rangle \leq c_{1}} \geq 1- \frac{c_0}{c_1}.
\]
Thus, by fixing $c_{1} > c_0/\eps$ and choosing the closure of $B_{c_{1}} \subseteq \mathcal{E}$ as the required compact set, we have the required compact containment condition \eqref{compact_cont}.

For the second criterion, we define an appropriate family of test functions $\mathbb{F}$, then apply Lemma~\ref{lem:ethier-kurtz}. In particular, we choose the family of functions $\mathbb{F}$ from $\mathcal{E}$ to $\mathbb{R}$ such that
\[
\mathbb{F} := \left\{\tilde{J}: \:  \tilde{J}(\mathbf{u}) = \int_{E} J(x) \mathbf{u}(\dd x ) 
; \: J \in C_{c}(E; \mathbb{R})\right\},
\]
where $C_{c}(E; \mathbb{R})$ denotes the set of continuous functions on $E$ with compact support. By the definition of the weak topology, this family consists of continuous functions and it is straightforward to see that it is closed under addition. Moreover, since $E$ is $\sigma$-compact, a measure $\mu$ is uniquely determined by the values of $\langle f, \mu \rangle$, where $f \in C_{c}(E; \mathbb{R})$, thus this family separates points. 
Now, let $\tilde{J} \in \mathbb{F}$ be given, with associated function $J: E \mapsto \mathbb{R}$, so that 
\begin{equation}\label{Jtilde}
\tilde{J}(\bar{\mathbf{L}}_{t}) = \langle J, \bar{\mathbf{L}}_{t} \rangle = \int_{E}J(x) \bar{\mathbf{L}}_{t}(\dd x).
\end{equation}
We seek to apply Lemma~\ref{lem:ethier-kurtz} to the family of pushforward measures 
\[
\left\{\tilde{J}_{*} \mathbb{P}_{N}: \; N \in \mathbb{N}\right\}.
\]
Note that these are measures on the space $D([0,\infty), \mathbb{R})$ which a separable, and complete metric space, hence Lemma~\ref{lem:ethier-kurtz} applies. Also note, that as the continuous image of a compact set is compact, we can take $\tilde{J}(\mathcal{E}_{c_{1}})$ as the compact set for the first condition, and thus we need only verify the second condition of Lemma~\ref{lem:ethier-kurtz}. 

We note that, for fixed $T$, and $\eta = \eta(T)$ sufficiently small, we can find an integer $K \in \mathbb{N}$ such that $\eta < T/K=:\eta' \leq 2\eta$. Therefore, we can define a partition $\{t_{i}\}$ of $[0,T]$ such that $t_{i+1} - t_{i} = \eta' > \eta$, so that
\begin{equation}
\begin{aligned} \label{eq:modified-modulus}
\tilde{J}_{*}(\mathbb{P}_{N})\left(\left\{f: w'(f, \eta, T) \geq \eps\right\} \right)
&= \mathbb{P}_{N}\left(w'((\tilde{J}(\bar{\mathbf{L}}_t))_{t\in[0,T]}, \eta, T) \geq \eps \right)
\\ &\leq \mathbb{P}_{N}\left(\sup_{s, t \in [0,T], |s-t| \leq \eta'} \left|\tilde{J}(\bar{\mathbf{L}}_s) - \tilde{J}(\bar{\mathbf{L}}_t)\right| \geq \eps\right),
\end{aligned}
\end{equation}
where $w'$ denotes the modulus of continuity defined in \eqref{mod_cont}.
We now have the following claim:
\begin{clm} \label{clm:tight-goal}
For some constant $C = C(T)$, we have
\begin{equation} \label{eq:tight-goal}
\limsup_{N \to \infty} \mathbb{E}_{N}\left[\sup_{s, t \in [0,T], |s-t| \leq \eta} \left|\tilde{J}(\bar{\mathbf{L}}_t) - \tilde{J}(\bar{\mathbf{L}}_s)\right|\right] < C\eta,
\end{equation}
\end{clm}
To complete the proof of Lemma~\ref{lem:tight} 
using Claim~\ref{clm:tight-goal}, observe that by Markov's inequality, for all $n \geq N$, we have 
\[
\lim_{\eta \to 0} \limsup_{N \to \infty} \mathbb{P}_{N}\left(\sup_{s, t \in [0,T], |s-t| \leq \eta} \left|\tilde{J}(\bar{\mathbf{L}}_t) - \tilde{J}(\bar{\mathbf{L}}_s)\right| \geq \eps\right) < \lim_{\eta \to 0} 
\frac{C\eta}{\eps} = 0, 
\]
implying~\eqref{tight_cond2}.
\end{proof}

\begin{proof}[Proof of Claim~\ref{clm:tight-goal}] 
We first note some relevant facts: since $\bar{\mathbf{L}}_t$ is a pure jump Markov process, and $\tilde{J}$ is bounded and measurable, it is well-known that
\begin{equation}\label{MG}
M_N(t):= \tilde{J}(\bar{\mathbf{L}}_t) - \tilde{J}(\bar{\mathbf{L}}_0) -\int_{0}^{t}\mathcal{A}_N\tilde{J}(\bar{\mathbf{L}}_s)\dd s 
\end{equation}
and its quadratic variation 
\begin{equation}\label{QV}
Q_N(t):= M_N(t)^{2}- \int_{0}^{t}(\mathcal{A}_N\tilde{J}^{2}-2\tilde{J}\mathcal{A}_N\tilde{J})(\bar{\mathbf{L}}_s) \dd s 
\end{equation}
are both martingales under $\mathbb{P}_{N}$ (see, for example, the proofs of~\cite[Proposition~7.1.6, and Proposition~8.3.3]{revuzyor}).  
From Equation~\eqref{MG}, the triangle inequality, and sub-additivity of taking suprema, it follows that 
\begin{linenomath}
\begin{align}
& \mathbb{E}_{N}\left[\sup_{s, t \in [0,T], |s-t| \leq \eta } \left|\tilde{J}[\bar{\mathbf{L}}(t)] - \tilde{J}[\bar{\mathbf{L}}_{s}]\right|\right] 
\\ & \leq \mathbb{E}_{N}\left[\sup_{s, t \in [0,T], |s-t| \leq \eta }\left|M_N(t) - M_N(s)\right|\right] + \mathbb{E}_{N}\left[\sup_{s, t \in [0,T], |s-t| \leq \eta }\left|\int_{s}^{t}\mathcal{A}_N\tilde{J}(\bar{\mathbf{L}}_{\theta})\dd\theta\right|\right]
\\ &  \leq 2 \mathbb{E}_{N}\left[\sup_{0 \leq t\leq T }\left|M_N(t)\right|\right] + \mathbb{E}_{N}\left[\sup_{s, t \in [0,T], t-s \leq \eta }\left|\int_{s}^{t}\mathcal{A}_N\tilde{J}(\bar{\mathbf{L}}_{\theta})\dd\theta\right|\right]. \label{ineq}
\end{align}
\end{linenomath}
To complete the proof, we bound the two terms on the right side of Equation~\eqref{ineq}. For the first, observe that by Doob's maximal inequality,
\[
\mathbb{E}_{N}\left[\left(\sup_{0\leq t\leq T} |M(t)|\right)^{p}\right] \leq \left(\frac{p}{p-1}\right)^{p}\mathbb{E}_{N}\left[|M(_NT)|^{p}\right],
\]
so that, by setting $p=2$, and recalling that $\mathbb{E}_{N}\left[Q(t)\right]=0$, we deduce from Equation~\eqref{QV} that
\begin{linenomath}
\begin{align}\label{Doob}
\mathbb{E}_{N}\left[\sup_{0\leq t\leq T} |M_N(t)|\right]^{2} &\leq \mathbb{E}_{N}\left[\left(\sup_{0\leq t\leq T} |M_N(t)|\right)^2\right] \\ & \leq 4\mathbb{E}_{N}\left[M_N(T)^{2}\right] = 4\mathbb{E}_{N}\left[\int_{0}^{T}(\mathcal{A}_N\tilde{J}^{2}-2\tilde{J}\mathcal{A}_N\tilde{J})(\bar{\mathbf{L}}_{s}) \dd s\right].
\end{align}
\end{linenomath}
Now, we recall that the generator $\mathcal{A}_{N}$ of the normalised process $(\bar{\mathbf{L}}^{(N)}_{t})_{t \geq 0}$ may be written as follows: for bounded measurable test functions $F$, we have 
\begin{linenomath}
\begin{align*}
& \mathcal{A}_{N}F(\bar{\mathbf{L}}^{(N)}_{t}) \\ &= \frac{N}{2}\int_{E\times E\times E} (\bar{\mathbf{L}}^{(N)}_{t}(\dd x) \left(\bar{\mathbf{L}}^{(N)}_{t} - \frac{\delta_{x}}{N}\right)(\dd y) K(x,y,\dd z)\\
&\mspace{300mu}\times\left(F\left(\bar{\mathbf{L}}^{(N)}_{t} + \frac{\left(\delta_{z} -\delta_{x} - \delta_{y} \right)}{N}\right) - F\left(\bar{\mathbf{L}}^{(N)}_{t}\right)\right).
\end{align*}
\end{linenomath}

For simplicity, for the remainder of this section, whenever unambiguous, we drop the superscript (or subscript) $(N)$ when referring to $\bar{\mathbf{L}}^{(N)}_{t}$, $M_N(t)$ and $\mathcal{A}_N$. Abusing notation, for each $t$ we denote by $\bar{{\mathbf{L}}}^{(x,y)\to z}_{t} := {\bar{\mathbf{L}}_{t}} + \left(\delta_{z} -\delta_{x} - \delta_{y}\right)/N $

Thus, 
\begin{linenomath}
\begin{align}
\nonumber &(\mathcal{A}\tilde{J}^{2}-2\tilde{J}\mathcal{A}\tilde{J})(\bar{\mathbf{L}}_s)
\\ \nonumber& = \frac{N}{2} \int_{E \times E \times E}\bar{\mathbf{L}}_{s}(\dd x) \left(\bar{\mathbf{L}}_{s} - \frac{\delta_{x}}{N}\right)(\dd y) K(x,y,\dd z)
\left( \tilde{J}(\bar{{\mathbf{L}}}^{(x,y)\to z}_{s})^{2} - \tilde{J}(\bar{\mathbf{L}})^2\right)
\\ \nonumber& \mspace{50mu}- N \tilde{J}(\bar{\mathbf{L}}_{s})\int_{E \times E \times E}\bar{\mathbf{L}}_{s}(\dd x) \left(\bar{\mathbf{L}}_{s} - \frac{\delta_{x}}{N}\right)(\dd y) K(x,y,\dd z)
\left(\tilde{J}(\bar{{\mathbf{L}}}^{(x,y)\to z}_{s}) - \tilde{J}(\bar{\mathbf{L}})\right)
\\ \nonumber& 
= \frac{N}{2} \int_{E \times E \times E}\bar{\mathbf{L}}_{s}(\dd x) \left(\bar{\mathbf{L}}_{s} - \frac{\delta_{x}}{N}\right)(\dd y) K(x,y,\dd z)
\left(\tilde{J}(\bar{{\mathbf{L}}_{s}}^{(x,y)\to z})^{2} - \tilde{J}(\bar{\mathbf{L}}_{s})^2\right)
\\ \nonumber& \mspace{30mu}- N \int_{E \times E \times E}\bar{\mathbf{L}}_{s}(\dd x) \left(\bar{\mathbf{L}}_{s} - \frac{\delta_{x}}{N}\right)(\dd y) K(x,y,\dd z)
\left(\tilde{J}(\bar{\mathbf{L}}_{s})\tilde{J}(\bar{{\mathbf{L}}}^{(x,y)\to z}_{s}) - \tilde{J}(\bar{\mathbf{L}}_{s})^2\right)
\\ & = \frac{N}{2} \int_{E \times E \times E}\bar{\mathbf{L}}_{s}(\dd x) \left(\bar{\mathbf{L}}_{s} - \frac{\delta_{x}}{N}\right)(\dd y) K(x,y,\dd z)
\left(\tilde{J}(\bar{{\mathbf{L}}}^{(x,y)\to z}_{s}) - \tilde{J}(\bar{\mathbf{L}}_{s})\right)^{2}, \label{eq:bound-quad-var}
\end{align}
\end{linenomath}
where, by definition \eqref{Jtilde} of $\tilde{J}$, we have 
\begin{equation} \label{eq:expl-def-J}
\tilde{J}(\bar{{\mathbf{L}}}^{(x,y)\to z}_{s}) - \tilde{J}(\bar{\mathbf{L}}_s) = \left(J(z) - J(y) - J(x)\right)/N. 
\end{equation}
Combining this with Equations~\eqref{eq:bound-quad-var} and~\eqref{Doob}, we get 
\begin{linenomath}
\begin{align}
&\mathbb{E}_{N}\left[\sup_{0\leq t\leq T} |M(t)|\right]^{2} 
\\ & \leq \mathbb{E}_{N} \bigg[\frac{2}{N} \int_{0}^{T}\dd s \int_{E \times E \times E}\bar{\mathbf{L}}_{s}(\dd x) \left(\bar{\mathbf{L}}_{s} - \frac{\delta_{x}}{N}\right)(\dd y) K(x,y,\dd z)
\left(J(z) - J(y) - J(x)\right)^{2} \bigg] 
\label{bound1}.
\end{align}
\end{linenomath}
Moreover, recalling that $J$ is continuous with compact support, by the extreme value theorem, it is bounded. 
Therefore, bounding $(J(z) - J(y) - J(x))^2$ by a constant $c_{J}$, and recalling that, by assumption, $\bar{K} \leq \phi'$ pointwise, we have
\begin{linenomath}
\begin{align} 
\mathbb{E}_{N}\left[\sup_{0\leq t\leq T} |M(t)|\right]^{2} & \leq \mathbb{E}_{N}\left[\frac{2c_J}{N} \int_{0}^{T} \dd s \int_{E \times E} \bar{\mathbf{L}}_{s}(\dd x) \left(\bar{\mathbf{L}}_{s}(\dd y) - \frac{\delta_{x}}{N}\right) \bar{K}(x, y)\right]
\\ & \leq \frac{2c_{J}}{N} \mathbb{E}_{N}\left[ \int_{0}^{T} \dd s \int_{E \times E} \bar{\mathbf{L}}_{s}(\dd x) \left(\bar{\mathbf{L}}_{s}(\dd y) - \frac{\delta_{x}}{N}\right) \phi'(x, y)\right]\\ \label{eq:bound1-tight}
&=\frac{2c_{J} T}{N} \mathbb{E}_{N}\left[ \int_{E \times E} \bar{\mathbf{L}}_{0}(\dd x) \left(\bar{\mathbf{L}}_{0}(\dd y) - \frac{\delta_{x}}{N}\right) \phi'(x, y)\right].
\end{align}
\end{linenomath}
The last step is possible since we now observe, that, as $\phi'$ is doubly sub-conservative, for each $s \in [0, \infty)$ we have 
\begin{linenomath}
\begin{align} \label{eq:couble-cons}
&\int_{E \times E} \bar{\mathbf{L}}^{(N)}_{s}(\dd x) \left(\bar{\mathbf{L}}^{(N)}_{s}(\dd y) - \frac{\delta_{x}}{N}\right) \phi'(x, y)\leq \int_{E \times E} \bar{\mathbf{L}}^{(N)}_{0}(\dd x) \left(\bar{\mathbf{L}}^{(N)}_{0}(\dd y) - \frac{\delta_{x}}{N}\right) \phi'(x, y) ,
\end{align}
\end{linenomath}
almost surely. 
Thus, by Equations~\eqref{eq:bound1-tight} and~\eqref{eq:couble-cons}, we have 
\begin{linenomath}
\begin{equation} \label{eq:bound2-tight}
\mathbb{E}_{N}\left[\sup_{0 \leq t\leq T }\left|M(t)\right|\right]\leq \sqrt{\frac{2c_{J}T}{N} \mathbb{E}_{N}\left[\int_{E \times E} \bar{\mathbf{L}}_{0}(\dd x) \left(\bar{\mathbf{L}}_{0}(\dd y) - \frac{\delta_{x}}{N}\right) \phi'(x, y)\right]}.
\end{equation}
\end{linenomath}
In order to bound the second term on the right-side of~\eqref{ineq}, we apply a similar argument: 
\begin{equation}
\begin{aligned}
&\left|\int_{s}^{t} \mathcal{A}\tilde{J}(\bar{\mathbf{L}}_{\theta}) \dd\theta \right| \\ & = \left|\frac{N}{2} \int_{s}^{t} \dd \theta \int_{E \times E \times E}\bar{\mathbf{L}}_{t}(\dd x) \left(\bar{\mathbf{L}}^{(N)}_{t} - \frac{\delta_{x}}{N}\right)(\dd y) K(x,y,\dd z) \tilde{J}(\bar{{\mathbf{L}}}^{(x,y)\to z}_{\theta}) - \tilde{J}(\bar{\mathbf{L}}_{\theta}) \right|
\\ & \leq \frac{N}{2} \int_{s}^{t} \dd \theta \int_{E \times E \times E}\bar{\mathbf{L}}_{t}(\dd x) \left(\bar{\mathbf{L}}_{t} - \frac{\delta_{x}}{N}\right)(\dd y) K(x,y,\dd z) \left|\tilde{J}(\bar{{\mathbf{L}}}^{(x,y)\to z}_{\theta}) - \tilde{J}(\bar{\mathbf{L}}_{\theta}) \right|
\end{aligned} 
\end{equation}
As before, using Equation~\eqref{eq:expl-def-J}, and the fact that $|x| = \sqrt{x^2}$, we may bound the previous by 
\begin{linenomath}
\begin{align}
&\frac{\sqrt{c_{J}}}{2} \int_{s}^{t} \dd \theta \int_{E \times E} \bar{\mathbf{L}}_{s}(\dd x) \left(\bar{\mathbf{L}}_{s}(\dd y) - \frac{\delta_{x}}{N}\right) \bar{K}(x, y) 
\\ & \hspace{1cm} 
\leq \frac{\sqrt{c_{J}}(t-s)}{2} \int_{E \times E} \bar{\mathbf{L}}_{0}(\dd x) \left(\bar{\mathbf{L}}_{0}(\dd y) - \frac{\delta_{x}}{N}\right) 
 \phi'(x, y),
\end{align}
\end{linenomath}
where the final inequality follows from~\eqref{eq:couble-cons}. 
Thus, we have obtained the upper bound 
\begin{linenomath}
\begin{align} \label{eq:second-bound}
&\mathbb{E}_{N}\left[\sup_{s, t \in [0,T], |s-t| \leq \eta }\left|\int_{s}^{t}\mathcal{A}\tilde{J}(\bar{\mathbf{L}}_{\theta})\dd\theta\right|\right]\leq 
\frac{\sqrt{c_{J}}}{2}\eta \mathbb{E}_{N}\left[\int_{E \times E} \bar{\mathbf{L}}_{0}(\dd x) \left(\bar{\mathbf{L}}_{0}(\dd y) - \frac{\delta_{x}}{N}\right) \phi'(x, y)\right].
\end{align}
\end{linenomath}
Thus, combining Equation~\eqref{ineq} with Equations~\eqref{eq:bound2-tight} and~\eqref{eq:second-bound}, and passing to the limit as $N \to \infty$,  
\begin{linenomath}
\begin{align} \label{eq:compact-supp-test-jump-bound}
& \limsup_{N \to \infty} \mathbb{E}_{N}\left[\sup_{s, t \in [0,T], |s-t| \leq \eta } \left|\tilde{J}(\bar{\mathbf{L}}_t) - \tilde{J}(\bar{\mathbf{L}}_{s})\right|\right] 
 \\ & \hspace{2cm} \leq \frac{\sqrt{c_{J}}}{2}\eta \limsup_{N \to \infty} \mathbb{E}_{N}\left[\int_{E \times E} \bar{\mathbf{L}}_{0}(\dd x) \left(\bar{\mathbf{L}}_{0}(\dd y) - \frac{\delta_{x}}{N}\right) \phi'(x, y)\right]
\end{align}
\end{linenomath}
the latter bound being finite by~\eqref{eq:conv-cons-prob}. 
Setting \[C:= \frac{\sqrt{c_{J}}}{2} \limsup_{N \to \infty} \mathbb{E}_{N}\left[\int_{E \times E} \bar{\mathbf{L}}_{0}(\dd x) \left(\bar{\mathbf{L}}_{0}(\dd y) - \frac{\delta_{x}}{N}\right)\phi'(x, y)\right]\] concludes the proof of~\eqref{eq:tight-goal}. 
\end{proof}

\subsubsection{Concentration of trajectories: proof of Item~\ref{item:conv-to-flory} of Theorem~\ref{spat_flory}} \label{sec:accumulation}
Assume, now, that Assumption~\ref{ass:conv-to-flory} is satisfied. This implies that Assumption~\ref{ass:tightness} is satisfied, so that, by Item~\ref{item:tightness} of Theorem~\ref{spat_flory}, the sequence of measures $(\mathbb{P}_{N})_{N \in \mathbb{N}}$ is tight. 

Let $\mathbb{P}^{*}$ denote an accumulation point of $(\mathbb{P}_{N})_{N \in \mathbb{N}}$, and assume, by passing to a subsequence, and re-indexing, that $\mathbb{P}_{N} \rightarrow \mathbb{P}^{*}$ with respect to the weak topology on the space of measures on $D([0, \infty); \mathcal{E})$. Following our previous notation, we denote by $\bar{\mathbf{L}}^{(N)}$ and $\bar{\mathbf{L}}^{*}$ random trajectories sampled from these distributions. Mostly out of convenience of notation, applying the Skorokhod representation theorem~\cite[Theorem~IV.13, page 71]{pollard},\footnote{Noting that as a tight probability measure on a metric space, $\mathbb{P}^{*}$ concentrates on a separable set.} we assume that $(\bar{\mathbf{L}}^{(N)})_{N \in \mathbb{N}}$ converges to $\bar{\mathbf{L}}^{*}$ pointwise for all $\omega \in \Omega$ with respect to the Skorokhod topology on $D([0, \infty); \mathcal{E})$ on some enlarged probability space $(\Omega, \mathscr{F}, \mathbb{P}(\cdot))$. For the rest of the section, we use the notation $\E{\cdot}$ to denote expectations with respect to this enlarged probability space. 
This allows us draw conclusions about the limiting trajectory $\bar{\mathbf{L}}^{*}$, and thus the limiting measure $\mathbb{P}_{*}$, more easily. We have the following proposition:  
\begin{prop} \label{prop:weak-conv-margin}
    For any $t \in [0, \infty)$ we have $\bar{\mathbf{L}}^{(N)}_t \rightarrow \bar{\mathbf{L}}^{*}_t$ almost surely in the weak topology. In addition, $\bar{\mathbf{L}}^{(N)}_t \otimes \bar{\mathbf{L}}^{(N)}_t\rightarrow \bar{\mathbf{L}}^{*}_t\otimes \bar{\mathbf{L}}^{*}_t$ almost surely in the weak topology, where the symbol  $\otimes$ denotes the product measure on the space $(\Omega, \mathscr{F}, \mathbb{P}(\cdot))$.
\end{prop}
\begin{proof}[Proof of Proposition~\ref{prop:weak-conv-margin}]
The proof is the result of the following observations:
\begin{enumerate}[label = (\Roman*)]
\item First note that, for any $J \in C_{c}(E; \mathbb{R})$, the operator $\tilde{J}: D([0, \infty); \mathcal{E}) \rightarrow D([0, \infty); \mathbb{R})$ is continuous, as the function $\tilde{J}:\mathcal{E} \rightarrow \mathbb{R}$ defined by $\tilde{J}(u) = \langle J, u \rangle$ is continuous (see for example,~\cite[Theorem~4.3]{jakubowski}). This implies that, if $\tilde{J}(\bar{\mathbf{L}}^{(N)})$  denotes the map $t \mapsto \tilde{J}(\bar{\mathbf{L}}^{(N)}_{t})$, for any $J \in C_{c}(E; \mathbb{R})$, we have $\tilde{J}(\bar{\mathbf{L}}^{(N)}) \rightarrow \tilde{J}(\bar{\mathbf{L}}^{*})$ almost surely in $D([0, \infty); \mathbb{R})$.
\item Applying~\eqref{eq:tight-goal}, and observing that $\sup_{s, t \in [0,T], |s-t| \leq \eta} \left|\tilde{J}(\bar{\mathbf{L}}_t) - \tilde{J}(\bar{\mathbf{L}}_s)\right|$ is monotone decreasing in $\eta$, we have
\begin{equation}
\lim_{N \to \infty} \mathbb{E}\left[\lim_{\eta\to 0}\sup_{s, t \in [0,T], |s-t| \leq \eta} \left|\tilde{J}(\bar{\mathbf{L}}_t) - \tilde{J}(\bar{\mathbf{L}}_s)\right|\right] =0.
\end{equation}
In addition, one may readily verify that, for any $T \in [0,\infty)$ the functional  
\begin{equation}\label{cont_fct}
x \mapsto \lim_{\eta \to 0} \sup_{s, t \in [0,T], |s-t| \leq \eta} \left\|x(t) - x(s)\right\|
\end{equation}
is a continuous functional with respect to the Skorokhod topology. Consequentially, by bounded convergence,
\begin{linenomath}
\begin{align*} 
0 &= \E{\lim_{N \to \infty} \lim_{\eta \to 0}\sup_{s, t \in [0,T], |s-t| \leq \eta } \left|\tilde{J}(\bar{\mathbf{L}}^{(N)}_t) - \tilde{J}(\bar{\mathbf{L}}^{(N)}_{s})\right|}\\ &\stackrel{\eqref{cont_fct}}{=}  \E{\lim_{\eta \to 0}\sup_{s, t \in [0,T], |s-t| \leq \eta } \left|\tilde{J}(\bar{\mathbf{L}}^{*}_t) - \tilde{J}(\bar{\mathbf{L}}^{*}_{s})\right|}
\end{align*}
\end{linenomath}
for any $J \in C_{c}(E; \mathbb{R})$, where in the final equality we have used the continuity of $\eqref{cont_fct}$. Therefore,  the function $\tilde{J}(\mathbf{\bar{L}}^{*}): [0, \infty) \rightarrow \mathbb{R}$ such that $t \mapsto \tilde{J}(\mathbf{\bar{L}}^{*}_{t})$ is almost surely continuous (i.e., $\tilde{J}(\mathbf{\bar{L}}^{*})~\in~C([0, \infty), \mathbb{R})$ almost surely). 
\item This continuity implies that for any sequence $(t_{n})_{n \in \mathbb{N}}$ such that $t_{n} \rightarrow t$, 
for any $J \in C_{c}(E; \mathbb{R})$ we have 
\[
\tilde{J}(\mathbf{\bar{L}}^{*}_{t_{n}}) = \int_{E} J(x) \mathbf{\bar{L}}^{*}_{t_{n}} (\dd x) \rightarrow \int_{E} J(x) \mathbf{\bar{L}}^{*}_{t} (\dd x) \quad \text{almost surely}.
\]
But, since, by assumption on the initial condition we have $\bar{\mathbf{L}}^{*}_{0} = \mu$, where $\mu$ denotes the limiting measure from~\eqref{eq:limit-init-cond}, and the dynamics of the process ensure that $\|\bar{\mathbf{L}}^{(N)}_{t}\|$ is non-increasing for each $N$, we readily verify that each for each $t \in [0, \infty)$ we have $\mathbf{\bar{L}}^{*}_{t}(E) \leq \|\mu\|$. Thus, by approximating any $F \in C_{b}(E; \mathbb{R})$ by compactly supported functions, for any $F \in C_{b}(E; \mathbb{R})$ we have 
\[
\int_{E} F(x) \mathbf{\bar{L}}^{*}_{t_{n}}(\dd x) \rightarrow \int_{E} F(x) \mathbf{\bar{L}}^{*}_{t} (\dd x) \quad \text{almost surely}.
\]
This implies that $\bar{\mathbf{L}}^{*}$ is, almost surely,  a continuous trajectory of measures, i.e., $\bar{\mathbf{L}}^{*} \in C([0,\infty); \mathcal{E})$.
\item It is well-known that in a Skorokhod space the projection map $\pi_{t}: D([0, \infty); E) \rightarrow E$ is a continuous functional at any trajectory $x \in D([0, \infty); E)$ for which $t$ is a continuity point. Since every $t \in [0, \infty)$ is a continuity point of $\bar{\mathbf{L}}^{*}_{t}$, this implies 
that for any $t \in [0, \infty)$ we have $\bar{\mathbf{L}}^{(N)}_{t} \rightarrow \bar{\mathbf{L}}^{*}_{t}$ almost surely in the weak topology, as required. Now, by a similar approach to the proof of Lemma~\ref{lem:tight}, the family of measures $\left\{\bar{\mathbf{L}}^{(N)}_{t} \otimes \bar{\mathbf{L}}^{(N)}_{t}, N \in \mathbb{N}\right\}$, is (almost surely) tight, and by assumption uniformly bounded in norm, thus almost surely relatively compact by~\cite[Theorem~IV.29, page 82]{pollard}; and any accumulation point must be $\bar{\mathbf{L}}^{*}_{t} \otimes \bar{\mathbf{L}}^{*}_{t}$. Thus,  
$\bar{\mathbf{L}}^{(N)}_{t} \otimes \bar{\mathbf{L}}^{(N)}_{t} \rightarrow \bar{\mathbf{L}}^{*}_{t} \otimes \bar{\mathbf{L}}^{*}_{t}$ almost surely. 
\end{enumerate}
\end{proof}

Now we are ready to prove Item~\ref{item:conv-to-flory} of Theorem~\ref{spat_flory}. Parts of the proof rely on equations from the proof of Lemma~\ref{lem:tight}, hence we recommend that the reader be acquainted with this proof first.
\begin{proof}[Proof of Item~\ref{item:conv-to-flory} of Theorem~\ref{spat_flory}]
In order to simplify some expressions, we make some shorthands. Recall that for any compactly supported $J \in C_{c}(E; \mathbb{R})$, we denote by $\tilde{J}$ the functional such that
\[
\tilde{J}(\bar{\mathbf{L}}^{(N)}_{s}) = \langle J, \bar{\mathbf{L}}^{(N)}_{s} \rangle = \int_{E} \bar{\mathbf{L}}^{(N)}_{s}(\dd x) J(x).
\]
We also define the following functionals:
\begin{linenomath}
\begin{align} \label{G+}
 G^{+}(\bar{\mathbf{L}}^{(N)}_{s}, J) 
 :=  \frac{1}{2}\int_{E \times E \times E} \bar{\mathbf{L}}^{(N)}_s(\dd x) \bar{\mathbf{L}}^{(N)}_s(\dd y) K(x, y, \dd z)J(z).
\end{align}
\end{linenomath}
We recall that, by assumption, there exists a continuous function $\phi:E \times E\mapsto \mathbb{R}^{+}$  that satisfies Equations~\eqref{eq:good-kernels}, \eqref{eq:lim-cons-kern} and~\eqref{eq:conv-cons-prob}. We then define $\hat{G}$ by
\begin{linenomath}
\begin{align}\label{Ghat}
\hat{G}(\bar{\mathbf{L}}^{(N)}_{s},J) :=
\int_{E \times E}\bar{\mathbf{L}}^{(N)}_s(\dd x) \bar{\mathbf{L}}^{(N)}_s(\dd y) \left[\bar{K}(x, y)  - \phi(x, y)\right]J(y).
\end{align}
\end{linenomath}
Finally, we define the functional 
\begin{linenomath}
\begin{align}\label{H}
H(\bar{\mathbf{L}}^{(N)}_s, J) &:=  \int_{E}\bar{\mathbf{L}}^{(N)}_s(\dd y) \left\langle \phi(\cdot, y), \bar{\mathbf{L}}^{(N)}_0 \right\rangle J(y)= \int_{E \times E}\bar{\mathbf{L}}^{(N)}_s(\dd y)\bar{\mathbf{L}}^{(N)}_0(\dd x) \phi(x, y) J(y).
\end{align}
\end{linenomath}
We now have the following claim:
\begin{clm} \label{clm:compact-test-equation}
Almost surely, for any $t \in [0, \infty)$, $J \in C_{c}(E; \mathbb{R})$ we have 
\begin{equation} \label{eq:compact-test-equation}
\tilde{J}(\bar{\mathbf{L}}^{*}_t) - \tilde{J}(\init) = \int_{0}^{t}  G^{+}(\bar{\mathbf{L}}^{*}_s, J) - \hat{G}(\bar{\mathbf{L}}^{*}_s, J) - H(\bar{\mathbf{L}}^{*}_s, J) \, \dd s. 
\end{equation}
\end{clm}

Note that the truth of Equation~\eqref{eq:compact-test-equation} for any $J \in C_{c}(E; \mathbb{R})$ implies that, almost surely, $(\bar{\mathbf{L}}^{*}_t)_{t\geq 0}$ satisfies~\eqref{eq:flory-test-function} in Definition~\ref{weak_sol}. Now, note that, as an application of~Proposition~\ref{prop:weak-conv-margin}, for any $t \geq 0$ we have
\begin{equation} \label{eq:weak-conv-without-diagonal}
\bar{\mathbf{L}}^{(N)}_{t} (\dd x) \left(\bar{\mathbf{L}}^{(N)}_{t} (\dd y)-\frac{\delta_x}N\right)\rightarrow \bar{\mathbf{L}}^{*}_{t} \otimes \bar{\mathbf{L}}^{*}_{t}
\end{equation}
almost surely in the weak topology. Recall that we have $\bar{K} \leq \phi'$, where $\phi'$ is doubly sub-conservative and continuous by the first assumption of Theorem~\ref{spat_flory}, and $\phi$ (which is also continuous, and $\phi'$ both satisfy~\eqref{eq:conv-cons-prob}. Thus, exploiting weak convergence, and Lemma~\ref{lem:conservation}, we 
deduce that $(\bar{\mathbf{L}}^{*}_t)_{t\geq 0}$ also satisfies Equations~\eqref{eq:flory-part-two} and~\eqref{eq:annoying-cons-assumption}, thus is a solution of the multi-type Flory equation in the sense given by Definition~\ref{weak_sol}. Hence the claim completes the proof of the theorem. 
\end{proof}
It thus suffices to prove the claim. 
\begin{proof}[Proof of Claim~\ref{clm:compact-test-equation}]
    First note that by recalling the martingale from Equation~\eqref{MG}, together with the bound from 
    ~\eqref{eq:bound2-tight} in the proof of Lemma~\ref{lem:tight}, we obtain
\begin{linenomath}
\begin{align}\label{eq:mart-zero}
&\lim_{N \to \infty} \mathbb{E}\Big[\Big|\tilde{J}(\bar{\mathbf{L}}^{(N)}_t) - \tilde{J}(\bar{\mathbf{L}}^{(N)}_0) \\
&- \frac{1}{2}\int_{0}^{t}\dd s \int_{E \times E \times E}\bar{\mathbf{L}}^{(N)}_s(\dd x)\left(\bar{\mathbf{L}}^{(N)}_s(\dd y)-\frac{\delta_x}N\right) K(x, y, \dd z) 
\times \left(J(z) - J(y) - J(x)\right)\Big|\Big]\\
&=0.
\end{align}
\end{linenomath}
Now, we define 
\begin{linenomath}
\begin{align} \label{eq:g-minus}
G_N^{-}(\bar{\mathbf{L}}^{(N)}_{s}, J) & := \frac{1}{2}\int_{E \times E} \bar{\mathbf{L}}^{(N)}_s(\dd x) \left(\bar{\mathbf{L}}^{(N)}_s(\dd y) -\frac{\delta_x}N\right)\bar{K}((x, y)(J(x) + J(y)),\\
 G^{+}_N(\bar{\mathbf{L}}^{(N)}_{s}, J) &
 :=  \frac{1}{2}\int_{E \times E \times E} \bar{\mathbf{L}}^{(N)}_s(\dd x) \left(\bar{\mathbf{L}}^{(N)}_s(\dd y) -\frac{\delta_x}N\right) K(x, y, \dd z)J(z),\\
\hat{G}_N(\bar{\mathbf{L}}^{(N)}_{s},J) &:=
\int_{E \times E}\bar{\mathbf{L}}^{(N)}_s(\dd x)  \left(\bar{\mathbf{L}}^{(N)}_s(\dd y) -\frac{\delta_x}N\right)\left[\bar{K}(x, y)  - \phi(x, y)\right]J(y).
\end{align}
\end{linenomath}
We may thus re-write the inner integral appearing in Equation~\eqref{eq:mart-zero} as $G_N^{+}(\bar{\mathbf{L}}^{(N)}_s, J) - G_N^{-}(\bar{\mathbf{L}}^{(N)}_s, J)$, so that
\begin{equation} \label{eq:mart-zero-gpm}
 \lim_{N \to \infty}\mathbb{E}\left[\left|\tilde{J}(\bar{\mathbf{L}}^{(N)}_t) - \tilde{J}(\bar{\mathbf{L}}^{(N)}_0) -
\int_{0}^{t} G_N^{+}(\bar{\mathbf{L}}^{(N)}_{s}, J) - G^{-}_N(\bar{\mathbf{L}}^{(N)}_{s}, J) \, \dd s \right| \right] = 0.
\end{equation}

Now, we seek to exploit the convergence of $\bar{\mathbf{L}}^{(N)}_s$ to $\bar{\mathbf{L}}^{*}_s$, but note that as the integrand appearing in Equation~\eqref{eq:g-minus} is in general unbounded, and $G^{-}$ is, in general, not continuous. However, it is possible to show that this functional coincides with a continuous functional on the trajectories $s \mapsto \bar{\mathbf{L}}^{(N)}_s$. 
Indeed, since $\phi$ is conservative, by Lemma~\ref{lem:conservation} the quantity $\left\langle \phi(\cdot, y), \bar{\mathbf{L}}^{(N)}_s \right\rangle$ is fixed, 
by adding and subtracting the term corresponding to $\int_{E \times E}\bar{\mathbf{L}}^{(N)}_s(\dd y)\bar{\mathbf{L}}^{(N)}_0(\dd x) \phi(x, y) J(y)$, we have 
\[
G_N^{-}(\bar{\mathbf{L}}^{(N)}_s, J) = \hat{G}_N(\bar{\mathbf{L}}^{(N)}_s, J) + H(\bar{\mathbf{L}}^{(N)}_s, J)-\mathcal{E}_N(\bar{\mathbf{L}}^{(N)}_s, J), 
\]
with 
\[\mathcal{E}_N(\bar{\mathbf{L}}^{(N)}_s, J):=\frac 1N\int_{E}\phi(x,x)J(x)\bar{\mathbf{L}}^{(N)}_s(\dd x).\]
Thus, re-writing Equation~\eqref{eq:mart-zero-gpm}, we get
\begin{linenomath}
\begin{align} \label{eq:mart-zero-rewrite}
    &\lim_{N \to \infty} \mathbb{E}\bigg[\bigg|\tilde{J}(\bar{\mathbf{L}}^{(N)}_t) - \tilde{J}(\bar{\mathbf{L}}^{(N)}_0) - \int_{0}^{t}  G^{+}_N(\bar{\mathbf{L}}^{(N)}_s, J) - \hat{G}_N(\bar{\mathbf{L}}^{(N)}_s, J) - H(\bar{\mathbf{L}}^{(N)}_s, J) +\mathcal{E}_N(\bar{\mathbf{L}}^{(N)}_s, J) \, \dd s \bigg|\bigg] =0.
\end{align}
\end{linenomath}
Now, in order to complete the proof of Equation~\eqref{eq:compact-test-equation}, we need to argue that we can pass the limit inside the expectation, and exploit weak convergence to replace the terms corresponding to $\bar{\mathbf{L}}^{(N)}$ with $\bar{\mathbf{L}}^{*}$.
We can pass the limit inside if the term $\hat{G}(\bar{\mathbf{L}}^{(N)}_s, J)$ was bounded, and then, need to argue continuity of the operators $\hat{G}$ and $H$.
Consequentially, we first approximate the functional $\hat{G}$ by truncations $(\hat{G}^{(k)})_{k \in \mathbb{N}}$, such that, with compact sets as defined in Equation~\eqref{eq:lim-cons-kern}
\begin{equation}
\begin{aligned}
& \hat{G}^{(k)}(\bar{\mathbf{L}}^{(N)}_s, J) := 
\int_{C_{k}} \bar{\mathbf{L}}^{(N)}_s(\dd x) \int_{E} \bar{\mathbf{L}}^{(N)}_s(\dd y) \left(\bar{K}(x, y)  - \phi(x, y)\right) J(y).   
\end{aligned}
\end{equation}
We finish the proof with another claim.
\begin{clm} \label{clm:compact-test-equation-flory}
Almost surely, for $s, t \in [0, \infty)$, $J \in C_{c}(E)$ we have
\begin{equation} \label{eq:goal-trunc-approx-flory1}
\lim_{k \to \infty} \limsup_{N \to \infty}\E{\left|\int_{0}^{t}\hat{G}_N(\bar{\mathbf{L}}^{(N)}_{s}, J) - \hat{G}^{(k)}(\bar{\mathbf{L}}^{(N)}_{s}, J) \, \dd s \right|}= 0, 
\end{equation}
\begin{equation}
 \label{eq:approx_G_+_N}
\limsup_{N \to \infty}\left(\E{\left|\int_{0}^{t}G^+_N(\bar{\mathbf{L}}^{(N)}_{s}, J) - {G}^+(\bar{\mathbf{L}}^{(N)}_{s}, J) \, \dd s \right|}
+ \E{\left|\int_{0}^{t}\mathcal{E}_N(\bar{\mathbf{L}}^{(N)}_{s}, J) \, \dd s \right|} \right)= 0,
\end{equation}
\begin{equation} \label{eq:goal-trunc-approx-flory2}
    \lim_{k \to \infty} \E{\left|\int_{0}^{t} \hat{G}(\bar{\mathbf{L}}^{*}_{s}, J) - \hat{G}^{(k)}(\bar{\mathbf{L}}^{*}_{s}, J) \, \dd s\right|} = 0,
\end{equation}
and 
\begin{equation} \label{eq:conv-H}
   \lim_{N \to \infty} H(\bar{\mathbf{L}}^{(N)}_s, J) = H(\bar{\mathbf{L}}^{*}_s, J) \quad \text{almost surely.} 
\end{equation}
\end{clm}

Indeed, if Equations~\eqref{eq:goal-trunc-approx-flory1}, \eqref{eq:approx_G_+_N},~\eqref{eq:goal-trunc-approx-flory2} and \eqref{eq:conv-H} are satisfied, by approximating  $\hat{G}$ by $\hat{G}^{(k)}$ (using the triangle inequality) in the second equality, and using bounded convergence for the third, we have
\begin{linenomath} 
\begin{align} \label{eq:trunc-works}
0 & = \lim_{N \to \infty} \mathbb{E}\bigg[\bigg|\tilde{J}(\bar{\mathbf{L}}^{(N)}_t) - \tilde{J}(\bar{\mathbf{L}}^{(N)}_0) - \int_{0}^{t}  G_N^{+}(\bar{\mathbf{L}}^{(N)}_s, J) - \hat{G}_N(\bar{\mathbf{L}}^{(N)}_s, J) - H_N(\bar{\mathbf{L}}^{(N)}_s, J)+\mathcal{E}_N(\bar{\mathbf{L}}^{(N)}_s, J) \, \dd s\bigg|\bigg]
\\ & \hspace{-0.35cm} \stackrel{\eqref{eq:goal-trunc-approx-flory1},\eqref{eq:approx_G_+_N}}{=} \lim_{k \to \infty}  \lim_{N \to \infty} \mathbb{E}\bigg[\bigg|\tilde{J}(\bar{\mathbf{L}}^{(N)}_t) - \tilde{J}(\bar{\mathbf{L}}^{(N)}_0) - \int_{0}^{t}  G^{+}(\bar{\mathbf{L}}^{(N)}_s, J) - \hat{G}^{(k)}(\bar{\mathbf{L}}^{(N)}_s, J) - H(\bar{\mathbf{L}}^{(N)}_s, J) \, \dd s\bigg|\bigg]
\\ &  = \lim_{k \to \infty}\mathbb{E}\bigg[\lim_{N \to \infty} \bigg|\tilde{J}(\bar{\mathbf{L}}^{(N)}_t) - \tilde{J}(\bar{\mathbf{L}}^{(N)}_0) - \int_{0}^{t}  G^{+}(\bar{\mathbf{L}}^{(N)}_s, J) - \hat{G}^{(k)}(\bar{\mathbf{L}}^{(N)}_s, J) - H(\bar{\mathbf{L}}^{(N)}_s, J) \, \dd s \bigg|\bigg]
\\ & \stackrel{\eqref{eq:conv-H}}{=} \lim_{k \to \infty}\mathbb{E}\bigg[\bigg|\tilde{J}(\bar{\mathbf{L}}^{*}_t) - \tilde{J}(\init)  - \int_{0}^{t}  G^{+}(\bar{\mathbf{L}}^{*}_s, J) - \hat{G}^{(k)}(\bar{\mathbf{L}}^{*}_s, J) - H(\bar{\mathbf{L}}^{*}_s, J) \, \dd s\bigg|\bigg]
 \\ & \stackrel{\eqref{eq:goal-trunc-approx-flory2}}{=} \E{ \left|\tilde{J}(\bar{\mathbf{L}}^{*}_t) - \tilde{J}(\init) - \int_{0}^{t}  G^{+}(\bar{\mathbf{L}}^{*}_s, J) - \hat{G}(\bar{\mathbf{L}}^{*}_s, J) - H(\bar{\mathbf{L}}^{*}_s, J) \, \dd s\right|}.
\end{align}
\end{linenomath}
\end{proof}
Finally, we finish the proof of Claim~\ref{clm:compact-test-equation-flory}.
\begin{proof}[Proof of Claim~\ref{clm:compact-test-equation-flory}]
Note that we have
\begin{linenomath}
\begin{align} \label{eq:trunc-app-1}
&  \E{\left|\int_{0}^{t}\hat{G}_N(\bar{\mathbf{L}}^{(N)}_{s}, J) - \hat{G}^{(k)}(\bar{\mathbf{L}}^{(N)}_{s}, J) \, \dd s \right|}
\\ & \leq \frac 1N\E{\left|\int_{0}^{t}\int_{C_{k}} \bar{\mathbf{L}}^{(N)}_{s}(\dd x) |\bar{K}(x,x)-\phi(x,x)||J(x)| \, \dd s \right|} \\
&\qquad +\E{\int_{0}^{t}\int_{C_{k}^{c}} \bar{\mathbf{L}}^{(N)}_{s}(\dd x) \int_{E}\left(\bar{\mathbf{L}}^{(N)}_{s}(\dd y) -\frac{\delta_x}N\right)\left|\bar{K}(x, y)  - \phi(x, y)\right| \left|J(y)\right| \dd s},
\end{align}
\end{linenomath} 
where we immediately see that \begin{equation}\label{eq:trunc-limit}
\limsup_{N\to\infty} \frac 1N\E{\left|\int_{0}^{t}\int_{C_{k}} \bar{\mathbf{L}}^{(N)}_{s}(\dd x) |\bar{K}(x,x)-\phi(x,x)||J(x)| \, \dd s \right|} =0,
\end{equation}
since $\bar{K}$, $\phi$ and $J$ are continuous, they are bounded on the support of $J$ and thus so is the expectation. Now, if Equation~\eqref{bounded-outside-compact} applies, then the integrand of the second term in~\eqref{eq:trunc-app-1} is bounded by some constant $c' > 0$, thus by~\eqref{eq:weak-conv-without-diagonal} and the Portmanteau theorem, we have
\begin{linenomath}
\begin{align}
& \limsup_{N \to \infty} \E{\int_{0}^{t}\int_{C_{k}^{c}} \bar{\mathbf{L}}^{(N)}_{s}(\dd x) \int_{E}\left(\bar{\mathbf{L}}^{(N)}_{s}(\dd y) -\frac{\delta_x}N\right)\left|\bar{K}(x, y)  - \phi(x, y)\right| \left|J(y)\right| \dd s} \\ & \hspace{3cm} \leq c' \E{\int_{0}^{t}\bar{\mathbf{L}}^{*}_{s}\left(\overline{C_{k}^{c}}\right) \bar{\mathbf{L}}^{*}_{s}(E)\dd s};
\end{align}
\end{linenomath}
and applying bounded convergence, using the fact that $\bigcap_{k \in \mathbb{N}} \overline{C^{c}_{k}} = \varnothing$, we deduce~\eqref{eq:goal-trunc-approx-flory1}. 
Otherwise, in the case that Equation~\eqref{eq:lim-cons-kern} applies, we bound the second term in~\eqref{eq:trunc-app-1} as follows:
\begin{linenomath}
\begin{align} \label{eq:trunc-app-2}
&\E{\int_{0}^{t}\int_{C_{k}^{c}} \bar{\mathbf{L}}^{(N)}_{s}(\dd x) \int_{E}\left(\bar{\mathbf{L}}^{(N)}_{s}(\dd y) -\frac{\delta_x}N\right)\left|\bar{K}(x, y)  - \phi(x, y)\right| \left|J(y)\right| \dd s},
\\ & \leq \E{\int_{0}^{t}\int_{C_{k}^{c}}  \bar{\mathbf{L}}^{(N)}_{s}(\dd x) \int_{E} \left(\bar{\mathbf{L}}^{(N)}_{s}(\dd y) -\frac{\delta_x}N\right) \phi^{*}(x, y) \, \dd s} \left\|J\right\|_{\infty} \sup_{x \in C_{k}^{c}, \, y \in \supp{(J)}}\left|\frac{\bar{K}(x, y)  - \phi(x, y)}{\phi^{*}(x, y)}\right|,
\end{align}
\end{linenomath} 
Since $\phi^{*}$ is doubly sub-conservative, and satisfies Equation~\eqref{eq:conv-cons-prob}, we have  \begin{linenomath}
\begin{align} \label{eq:limsup-bounded}
&     \limsup_{N \to \infty} \E{\int_{0}^{t}\int_{C_{k}^{c} \times E}  \bar{\mathbf{L}}^{(N)}_{s}(\dd x) \left(\bar{\mathbf{L}}^{(N)}_{s}(\dd y) -\frac{\delta_x}N\right)  \phi^{*}(x, y) \, \dd s} 
\\ & \leq \limsup_{N \to \infty} \, \E{\int_{0}^{t}\int_{E \times E}  \bar{\mathbf{L}}^{(N)}_{s}(\dd x) \left(\bar{\mathbf{L}}^{(N)}_{s}(\dd y) -\frac{\delta_x}N\right)  \phi^{*}(x, y) \, \dd s}
\\ & \leq \limsup_{N \to \infty} t \, \E{\int_{E \times E}  \bar{\mathbf{L}}^{(N)}_{0}(\dd x) \left(\bar{\mathbf{L}}^{(N)}_{0}(\dd y) -\frac{\delta_x}N\right) \phi^{*}(x, y)} \stackrel{\eqref{eq:conv-cons-prob}}{<} \infty.
\end{align}
\end{linenomath}
Combining Equation~\eqref{eq:limsup-bounded}, with \eqref{eq:trunc-app-1}, \eqref{eq:trunc-limit} and \eqref{eq:trunc-app-2}, we deduce Equation~\eqref{eq:goal-trunc-approx-flory1}.

Equation~\eqref{eq:approx_G_+_N} is proved in an analogous manner to~\eqref{eq:trunc-limit}, exploiting the compact support of $J$. 
For~\eqref{eq:goal-trunc-approx-flory2}, by the monotone convergence theorem,~\eqref{eq:weak-conv-without-diagonal} and Fatou's lemma we have 
\begin{linenomath}
    \begin{align} \label{eq:limsup-bounded-2}
        & \E{\int_{0}^{t}\int_{E \times E}  \bar{\mathbf{L}}^{*}_{s}(\dd x) \bar{\mathbf{L}}^{*}_{s}(\dd y) \phi^{*}(x, y) \, \dd s} = \E{\int_{0}^{t} \lim_{j\to \infty} \int_{E \times E}  \bar{\mathbf{L}}^{*}_{s}(\dd x) \bar{\mathbf{L}}^{*}_{s}(\dd y) \left( \phi^{*}(x, y) \wedge j \right)\, \dd s}
        \\ & \hspace{1cm} = \E{\int_{0}^{t} \lim_{j\to \infty} \lim_{N \to \infty} \int_{E \times E}  \bar{\mathbf{L}}^{(N)}_{s}(\dd x) \left(\bar{\mathbf{L}}^{(N)}_{s}(\dd y)-\frac{\delta_x}N\right) \left(\phi^{*}(x, y) \wedge j \right) \, \dd s}
        \\ & \hspace{1cm} \leq \limsup_{N \to \infty } \E{\int_{0}^{t}\int_{E \times E}  \bar{\mathbf{L}}^{(N)}_{s}(\dd x) \left(\bar{\mathbf{L}}^{(N)}_{s}(\dd y)-\frac{\delta_x}N\right) \phi^{*}(x, y) \, \dd s} \stackrel{\eqref{eq:limsup-bounded}}{<} \infty.
    \end{align}
\end{linenomath}
 Applying Equation~\eqref{eq:limsup-bounded-2} we deduce Equation~\eqref{eq:goal-trunc-approx-flory2} in a similar manner to Equation~\eqref{eq:goal-trunc-approx-flory1}. Finally, recalling the definition of the functional $H$ from Equation~\eqref{H}, we have 
 \[
\lim_{N \to \infty} H(\bar{\mathbf{L}}^{(N)}_s, J) = \lim_{N\to \infty} \int_{E \times E}\bar{\mathbf{L}}^{(N)}_s(\dd y) \bar{\mathbf{L}}^{(N)}_0(\dd x) \phi(x, y)J(y).
 \]
Thus,
\begin{linenomath}
    \begin{align} \label{eq:H-approx-bound} 
       & \limsup_{N \to \infty} \left| H(\bar{\mathbf{L}}^{(N)}_s, J) - H(\bar{\mathbf{L}}^{*}_s, J)\right| \\ & \hspace{0.25cm} = \limsup_{N \to \infty} 
        \left|\int_{E \times E}\bar{\mathbf{L}}^{(N)}_s(\dd y) \bar{\mathbf{L}}^{(N)}_0(\dd x) \phi(x, y)J(y) - \int_{E \times E}\bar{\mathbf{L}}^{*}_s(\dd y) \init(\dd x) \phi(x, y)J(y) \right| 
        \\& \hspace{0.25cm}\leq \limsup_{N \to \infty} 
        \left|\int_{E \times E}\bar{\mathbf{L}}^{(N)}_s(\dd y) \bar{\mathbf{L}}^{(N)}_0(\dd x) \phi(x, y)J(y) - \int_{E \times E}\bar{\mathbf{L}}^{(N)}_s(\dd y) \init(\dd x) \phi(x, y)J(y) \right| 
        \\ & \hspace{1cm} + \limsup_{N \to \infty} 
        \left|\int_{E \times E}\bar{\mathbf{L}}^{(N)}_s(\dd y) \init(\dd x) \phi(x, y)J(y) - \int_{E \times E}\bar{\mathbf{L}}^{*}_s(\dd y) \init(\dd x) \phi(x, y)J(y) \right|.
    \end{align}
\end{linenomath}
The second term in the upper bound of~\eqref{eq:H-approx-bound} is $0$,  since the map $y \mapsto \int_{E} \mu(\dd x) \phi(x,y) J(y)$ is bounded and continuous (because $J$ has compact support and $\phi$ is continuous), and $\bar{\mathbf{L}}^{(N)}_{s} \rightarrow \bar{\mathbf{L}}^{*}_{s}$. On the other hand, by applying Equation~\eqref{eq:conv-probab}, with the compact set $C'$ chosen to be the support of $J$, for any $\eps > 0$, there exists $N_{0}$ such that, for all $N \geq N_{0}$ we have 
\begin{linenomath}
\begin{align*}
       & \left|\int_{E \times E}\bar{\mathbf{L}}^{(N)}_s(\dd y) \bar{\mathbf{L}}^{(N)}_0(\dd x) \phi(x, y)J(y) - \int_{E \times E}\bar{\mathbf{L}}^{(N)}_s(\dd y) \init(\dd x) \phi(x, y)J(y) \right| 
       \\ & \hspace{4.5cm} \leq \eps\left|\int_{E}\bar{\mathbf{L}}^{(N)}_s(\dd y)  J(y) \right|,
\end{align*}
\end{linenomath}
and thus, taking limits superior of both sides, since $J$ is bounded and continuous, 
\begin{linenomath}
    \begin{align*}
       &  \limsup_{N \to \infty} 
        \left|\int_{E \times E}\bar{\mathbf{L}}^{(N)}_s(\dd y) \bar{\mathbf{L}}^{(N)}_0(\dd x) \phi(x, y)J(y) - \int_{E \times E}\bar{\mathbf{L}}^{(N)}_s(\dd y) \init(\dd x) \phi(x, y)J(y) \right| 
        \\ & \hspace{5cm} \leq \eps \left|\int_{E}\bar{\mathbf{L}}^{*}_s(\dd y) J(y) \right|.
    \end{align*}
\end{linenomath}
Sending $\eps \to 0$, we deduce that the first term in the upper bound of~\eqref{eq:H-approx-bound} is also $0$, hence conclude the proof of~\eqref{eq:conv-H}. 
\end{proof}
\subsection{Uniqueness: proof of Theorem~\ref{thm:uniqueness}}
\begin{proof}[Proof of Theorem~\ref{thm:uniqueness}]
Suppose that $(\mu_{s})_{s \geq 0}$ and $(\hat{\mu}_{s})_{s \geq 0}$ denote two solutions to the Flory equation, with a given initial condition $\mu_{0}$, with $\|\mu_{0}\| < \infty$.
Suppose that $\left(\mu_{s} - \hat{\mu}_{s}\right)|_{D_{R}}$ denotes the measure $\left(\mu_{s} - \hat{\mu}_{s}\right)$ restricted to $D_{R}$. 
By a well-known property of the total variation distance, we may write
\[
\left\|\left(\mu_{s} - \hat{\mu}_{s}\right)|_{D_{R}}\right\|
= \sup_{f: \|f\|_{\infty} = 1} \pairing{f \mathbf{1}_{D_{R}}, \left(\mu_{s} - \hat{\mu}_{s}) \right)}. 
\]
Note that, by a straightforward approximation argument (approximating a measurable function pointwise by continuous functions), if $(\mu_{t})_{t \geq 0}$ is a solution to the multi-type Flory equation as in Definition~\ref{weak_sol}, Equation~\eqref{eq:flory-test-function} is satisfied for all bounded measurable functions supported on the compact set $D_{k}$, for $k \in \mathbb{N}$.
Thus, for $f$ such that $\|f\|_{\infty} = 1$, we have
\begin{linenomath}
\begin{align} \label{eq:unique-flory}
&  \pairing{f \mathbf{1}_{D_{R}}, \left(\mu_{s} - \hat{\mu}_{s}) \right)}   
\\ & \hspace{1cm} = \frac{1}{2}\int_{0}^{s} \int_{E \times E \times E} f(z) \mathbf{1}_{D_{R}}(z) K(u, v,\dd z) \left(\mu_{r}(\dd u)  \mu_{r}(\dd v) - \hat{\mu}_{r}(\dd u)  \hat{\mu}_{r}(\dd v)  \right) \dd r
\\ & \hspace{2cm} - \int_{0}^{s} \int_{E \times E } f(u)  \mathbf{1}_{D_{R}}(u) \bar{K}(u,v) \left(\mu_{r}(\dd u)  \mu_{r}(\dd v) - \hat{\mu}_{r}(\dd u)  \hat{\mu}_{r}(\dd v)  \right) \dd r
\\ &  \hspace{2cm} - \int_{0}^{s} \int_{E \times E } f(u)  \mathbf{1}_{D_{R}}(u) \phi(u,v) \left(\mu_{r}(\dd u)  \mu_{0}(\dd v) - \hat{\mu}_{r}(\dd u)  \mu_{0}(\dd v)  \right) \dd r
\\ & \hspace{2cm} + \int_{0}^{s} \int_{E \times E } f(u)  \mathbf{1}_{D_{R}}(u) \phi(u,v) \left(\mu_{r}(\dd u)  \mu_{r}(\dd v) - \hat{\mu}_{r}(\dd u)  \hat{\mu}_{r}(\dd v)  \right) \dd r.
\end{align}
\end{linenomath}
We now bound the values of each of the terms in the above display. For the first, since $\phi$ is conservative, recalling the definition of $D_{R}$ in \eqref{D}, we have $\mathbf{1}_{D_{R}}(z) \leq \mathbf{1}_{D_{R}}(u) \mathbf{1}_{D_{R}}(v)$ for $K(u, v, \dd z)$-a.a. $z$. Moreover, bounding $f(z)$ above by $1$, we obtain 
\begin{linenomath}
    \begin{align} \label{eq:uniq-flory-plus-bound}
        &\nonumber \frac{1}{2}\int_{0}^{s} \int_{E \times E \times E} f(z) \mathbf{1}_{D_{R}}(z) K(u, v, \dd z) \left(\mu_{r}(\dd u)  \mu_{r}(\dd v) - \hat{\mu}_{r}(\dd u)  \hat{\mu}_{r}(\dd v)  \right) \dd r
        \\ & \leq \frac{1}{2}\int_{0}^{s} \int_{E \times E} \mathbf{1}_{D_{R}}(u) \mathbf{1}_{D_{R}}(v) \bar{K}(u,v) \left|\mu_{r}(\dd u)  \mu_{r}(\dd v) - \hat{\mu}_{r}(\dd u)  \hat{\mu}_{r}(\dd v)  \right| \dd r.
        \end{align}
\end{linenomath}
We now have the following claim:
\begin{clm} \label{clm:cross-terms}
    We have
    \begin{linenomath}
    \begin{align} \label{eq:clm-cross-terms}
    & \int_{E \times E} \mathbf{1}_{D_{R}}(u) \mathbf{1}_{D_{R}}(v) \phi(u,v) \left|\mu_{r}(\dd u)  \mu_{r}(\dd v) - \hat{\mu}_{r}(\dd u)  \hat{\mu}_{r}(\dd v)  \right| 
    \\ & \hspace{3cm} \leq 
     2 R\int_{E} \mathbf{1}_{D_{R}}(v) \left|\mu_{r}(\dd v) - \hat{\mu}_{r}(\dd v)  \right|. 
    \end{align}
    \end{linenomath}
\end{clm}

By applying Claim~\ref{clm:cross-terms}, and bounding $\bar{K}(x,y) \leq c' \phi(x,y)$, we may now bound the right-side of~\eqref{eq:uniq-flory-plus-bound}: 
\begin{linenomath}
    \begin{align} \label{eq:flory-unique-first}
        & \frac{1}{2}\int_{0}^{s} \int_{E \times E} \mathbf{1}_{D_{R}}(u) \mathbf{1}_{D_{R}}(v) \bar{K}(u,v) \left|\mu_{r}(\dd u)  \mu_{r}(\dd v) - \hat{\mu}_{r}(\dd u)  \hat{\mu}_{r}(\dd v)  \right| \dd r
        \\ & \hspace{1cm} \leq \frac{c'}{2}\int_{0}^{s} \int_{E \times E} \mathbf{1}_{D_{R}}(u) \mathbf{1}_{D_{R}}(v) \phi(u,v) \left|\mu_{r}(\dd u)  \mu_{r}(\dd v) - \hat{\mu}_{r}(\dd u)  \hat{\mu}_{r}(\dd v)  \right| \dd r
        \\ & \hspace{1cm} \leq c'\int_{0}^{s} R\int_{E} \mathbf{1}_{D_{R}}(v) \left|\mu_{r}(\dd v) - \hat{\mu}_{r}(\dd v)  \right| \dd r \leq c' R \int_{0}^{s} \left \|\left(\mu_{r} - \hat{\mu}_{r}\right)|_{D_{R}} \right\| \dd r. 
    \end{align}
\end{linenomath}
Next, re-writing, and combining the second and fourth terms in~\eqref{eq:unique-flory}, recalling that $\phi(x,y)$ coincides with $\bar{K}(x,y)$ on $(D_{R} \times D_{R})^{c}$ (so in particular $D_{R} \times D_{R}^{c}$) we get 
\begin{linenomath}
    \begin{align} 
        & \int_{0}^{s} \int_{E \times E} f(u) \mathbf{1}_{D_{R}}(u)  \left(\phi(u,v) - \bar{K}(u,v)\right) \left(\mu_{r}(\dd u)  \mu_{r}(\dd v) - \hat{\mu}_{r}(\dd u)  \hat{\mu}_{r}(\dd v)  \right) \dd r
        \\ & = \int_{0}^{s} \int_{E \times E} f(u) \mathbf{1}_{D_{R}}(u) \mathbf{1}_{D_{R}}(v)  \left(\phi(u, v) - \bar{K}(u,v)\right)\\ &\hspace{7cm}\times \left(\mu_{r}(\dd u)  \mu_{r}(\dd v) - \hat{\mu}_{r}(\dd u)  \hat{\mu}_{r}(\dd v)  \right) \dd r
        \\ & \hspace{1cm} \leq (c'+1) \int_{0}^{s} \int_{E \times E} \phi(u,v) \mathbf{1}_{D_{R}}(u) \mathbf{1}_{D_{R}}(v)\left|\mu_{r}(\dd u)  \mu_{r}(\dd v) - \hat{\mu}_{r}(\dd u)  \hat{\mu}_{r}(\dd v)\right| \dd r
        \\ & \hspace{1cm} \stackrel{\eqref{eq:clm-cross-terms}}{\leq} 2 R (c'+1) \int_{0}^{s} \int_{E} \mathbf{1}_{D_{R}}(v) \left|\mu_{r}(\dd v) - \hat{\mu}_{r}(\dd v) \right| \dd r 
        \\ &\hspace{1cm} = 2 R (c'+1) \int_{0}^{s} \left \|\left(\mu_{r} - \hat{\mu}_{r}\right)|_{D_{R}} \right\| \dd r, \label{eq:flory-unique-second-fourth}
    \end{align}
\end{linenomath}
where in the second to last inequality, we use the bound $\bar{K}(u,v) \leq c' \phi(u,v)$. Finally, for the third term in~\eqref{eq:unique-flory}, observing that $\mu_{0}$ is a positive measure, we make a similar computation: 
\begin{linenomath}
    \begin{align} 
        & - \int_{0}^{s} \int_{E \times E } f(u)  \mathbf{1}_{D_{R}}(u) \phi(u,v) \left(\mu_{r}(\dd u)  \mu_{0}(\dd v) - \hat{\mu}_{r}(\dd u)  \mu_{0}(\dd v)  \right) \dd r \\ & \hspace{3cm} \leq \int_{0}^{s} \int_{E \times E }  \mathbf{1}_{D_{R}}(u) \phi(u,v) \left| \mu_{r}(\dd u)  \mu_{0}(\dd v) - \hat{\mu}_{r}(\dd u)  \mu_{0}(\dd v)  \right| \dd r
        \\ & \hspace{3cm}  
        = \int_{0}^{s} \int_{E}  \mathbf{1}_{D_{R}}(u) \int_{E} \phi(u,v) \mu_{0}(\dd v) \left| \mu_{r}(\dd u) - \hat{\mu}_{r}(\dd u) \right| \dd r\\
        & \hspace{3cm} \leq R \int_{0}^{s} \left \|\left(\mu_{r} - \hat{\mu}_{r}\right)|_{D_{R}} \right\| \dd r \label{eq:flory-unique-third}
    \end{align}
\end{linenomath}
Combining Equations~\eqref{eq:flory-unique-first},~\eqref{eq:flory-unique-second-fourth} and~\eqref{eq:flory-unique-third}, to bound~\eqref{eq:unique-flory} we deduce that 
\begin{linenomath}
    \begin{align*}
       \left \|\left(\mu_{s} - \hat{\mu}_{s}\right)|_{D_{R}}  \right \| = \sup_{f: \|f\|_{\infty} = 1} \pairing{f \mathbf{1}_{D_{R}}, \left(\mu_{s} - \hat{\mu}_{s})\right)}
       \leq 3R(c' + 1)  \int_{0}^{s} \left \|\left(\mu_{r} - \hat{\mu}_{r}\right)|_{D_{R}} \right\| \dd r.
    \end{align*}
\end{linenomath}
\begin{clm} \label{clm:conservation}
    Suppose that $(\mu_{t})_{t \geq 0}$ is a solution to the multi-type Flory equation. Then, if $\xi: E \rightarrow \mathbb{R}_{+}$  is a sub-conservative function, for each $t \geq 0$
    \begin{equation} \label{eq:clm-conservation}
  \int_{E} \xi(x) \mu_{t}(\dd x) \leq \int_{E} \xi(x) \mu_{0}(\dd x). 
    \end{equation}
\end{clm}
By Claim~\ref{clm:conservation} applied to the function $\xi(x) \equiv 1$, we know that for each $s \geq 0$, we have 
\[\left \|\left(\mu_{s} - \hat{\mu}_{s}\right)|_{D_{R}}  \right \| \leq 2 \| \mu_0 \|.\] 
We can thus apply Gronwall's lemma to deduce that  $\left \|\left(\mu_{s} - \hat{\mu}_{s}\right)|_{D_{R}}  \right \| = 0$. As $\bigcup_{k \in \mathbb{N}} D_{k} = E$, it must be the case that $\left \|\mu_{s} - \hat{\mu}_{s}\right \| = 0$, showing uniqueness. 
\end{proof}
We finish with the proofs of~Claim~\ref{clm:cross-terms} and Claim~\ref{clm:conservation}:
\begin{proof}[Proof of Claim~\ref{clm:cross-terms}]
    We bound 
    \begin{linenomath}
        \begin{align} \label{eq:claim-1-unique-flory}
            & \int_{E \times E} \mathbf{1}_{D_{R}}(u)\mathbf{1}_{D_{R}}(v) \phi(u,v) \left|\mu_{r}(\dd u)  \mu_{r}(\dd v) - \hat{\mu}_{r}(\dd u)  \hat{\mu}_{r}(\dd v)  \right| 
            \\ & \hspace{1.5cm} \leq \int_{E \times E} \mathbf{1}_{D_{R}}(u)\mathbf{1}_{D_{R}}(v) \phi(u,v) \left|\mu_{r}(\dd u)  \mu_{r}(\dd v) - \mu_{r}(\dd u)  \hat{\mu}_{r}(\dd v)  \right| 
            \\ & \hspace{3cm} + \int_{E \times E} \mathbf{1}_{D_{R}}(u)\mathbf{1}_{D_{R}}(v)\phi(u,v) \left|\mu_{r}(\dd u)  \hat{\mu}_{r}(\dd v) - \hat{\mu}_{r}(\dd u)  \hat{\mu}_{r}(\dd v)  \right|.
        \end{align}
    \end{linenomath}    
    For the first term on the right-hand side of~\eqref{eq:claim-1-unique-flory}, we integrate the variable $u$, applying~\eqref{eq:annoying-cons-assumption} 
\begin{linenomath}
    \begin{align*}
        & \int_{E \times E} \mathbf{1}_{D_{R}}(u)\mathbf{1}_{D_{R}}(v)\phi(u,v) \left|\mu_{r}(\dd u)  \mu_{r}(\dd v) - \mu_{r}(\dd u)  \hat{\mu}_{r}(\dd v)  \right| \dd r \\ & \hspace{1cm} \leq \int_{E} \mathbf{1}_{D_{R}}(v) \int_{E} \phi(u,v) \mu_{r}(\dd u) \left|\mu_{r}(\dd v) - \hat{\mu}_{r}(\dd v)  \right| \dd r 
        \\ & \hspace{1cm} \leq \int_{E} \mathbf{1}_{D_{R}}(v) \int_{E} \phi(u,v) \mu_{0}(\dd u) \left|\mu_{r}(\dd v) - \hat{\mu}_{r}(\dd v)  \right| \dd r
        \\ & \hspace{1cm} \leq R \int_{E} \mathbf{1}_{D_{R}}(v) \left|\mu_{r}(\dd v) - \hat{\mu}_{r}(\dd v)  \right| \dd r;
    \end{align*}
\end{linenomath}
and applying a similar argument for the second term in~\eqref{eq:claim-1-unique-flory}, we deduce the result. 
\end{proof}
\begin{proof}[Proof of Claim~\ref{clm:conservation}]
 Suppose first that $\xi$ is bounded.  Then, applying~\eqref{eq:flory-test-function} to the function $\xi \mathbf{1}_{D_{k}}$, we have 
\begin{linenomath}
\begin{align} \label{eq:sub-cons-non-increasing}
    & \pairing{\xi \mathbf{1}_{D_{k}}, \mu_{s} - \mu_{0}} 
    \\ & = \frac{1}{2}\int_{0}^{s} \int_{E \times E \times E} \left(\xi(z) \mathbf{1}_{D_{k}}(z) - \xi(u) \mathbf{1}_{D_{k}}(u) - \xi(v) \mathbf{1}_{D_{k}}(v) \right)K(u, v,\dd z) \mu_{r}(\dd u)  \mu_{r}(\dd v) \dd r  
    \\ & \hspace{3cm} + \int_{0}^{s} \int_{E \times E } \xi(u)  \mathbf{1}_{D_{k}}(u) \phi(u,v) \left(\mu_{r}(\dd u)  (\mu_{r}(\dd v) - \mu_{0}(\dd v))  \right) \dd r.
\end{align}
\end{linenomath}
By~\eqref{eq:annoying-cons-assumption}, since $\xi(u) \geq 0$, and $\mu_{r}$ is a positive measure, we deduce that the second term on the right-side above is non-positive. In addition, since $\phi$ is conservative, so is the function $x \mapsto \pairing{\phi(x, \cdot), \mu_0}$, and we deduce that $\mathbf{1}_{D_{R}}(z) \leq \mathbf{1}_{D_{R}}(u) \mathbf{1}_{D_{R}}(v)$ for $K(u, v, \dd z)$-a.a. $z$. Finally, since $\xi$ is sub-conservative,  
\begin{linenomath}
\begin{align*}
& \int_{E \times E \times E} \left(\xi(z) \mathbf{1}_{D_{k}}(z) \right)K(u, v,\dd z) \mu_{r}(\dd u)  \mu_{r}(\dd v)
\\ & \hspace{2cm} \leq \int_{E \times E \times E} \left(\xi(u) + \xi(v)\right) \mathbf{1}_{D_{k}}(u) \mathbf{1}_{D_{k}}(u) K(u, v,\dd z) \mu_{r}(\dd u)  \mu_{r}(\dd v)
\\ & \hspace{2cm} \leq \int_{E \times E \times E} \left(\xi(u)\mathbf{1}_{D_{k}}(u) + \xi(v)\mathbf{1}_{D_{k}}(v)\right)   K(u, v,\dd z) \mu_{r}(\dd u)  \mu_{r}(\dd v).
\end{align*}
\end{linenomath}
Thus, $\pairing{\xi \mathbf{1}_{D_{k}}, \mu_{s}} \leq   \pairing{\xi \mathbf{1}_{D_{k}}, \mu_{0}}$, and we deduce the result from monotone convergence. Finally, we can extend the result to unbounded $\xi$, again from monotone convergence (approximating $\xi$ from below by the sub-conservative functions $\xi \wedge j$, for $j \in \mathbb{N}$). 
\end{proof}

\appendix
\section{General criteria for relative compactness} \label{app:tightness-criteria}
Recall that for each $N$, the cluster coagulation process $(\bar{\mathbf{L}}^{(N)}_{t})_{t \in [0, \infty)}$ is defined as taking values in the space
\[
\Ecal=\bigcup_{n\in\N} \left\{\mathbf{u} \in \mathcal{M}_{+}(\mathcal{E}): \langle m, \mathbf{u} \rangle \leq n \right\}
\]
Recall also that equip $\mathcal{E}$ with the Prokhorov metric, which metrises the topology of weak convergence. 
We may interpret $(\bar{\mathbf{L}}_{t})_{t \in [0, \infty)}$ as a trajectory in $D([0,\infty);\mathcal{E})$, the space of right-continuous functions $f: [0,\infty) \rightarrow \mathcal{E}$ with left-limits. We equip $D([0,\infty);\mathcal{E})$ with the \emph{Skorokhod metric} $d$. 
Recall that for a separable, complete metric space $(\mathcal{E},\delta)$ with $q := \delta \wedge 1$, the Skorokhod metric on $D([0, \infty); \mathcal{E})$ is defined as follows:
Let $\Lambda$ denotes the set of all strictly increasing functions mapping $[0, \infty)$ onto $[0, \infty)$, and $\Lambda' \subseteq \Lambda$ the subset of Lipschitz functions. Then, for $\lambda \in \Lambda'$, define 
\[
\gamma(\lambda) := \sup_{s > t \geq 0} \left|\log{\frac{\lambda(s) - \lambda(t)}{s-t} }\right| < \infty
\]
Then, for $f,g \in D([0, \infty); \mathcal{E})$, we define 
\begin{equation}\label{mod_cont}
d(f,g) := \inf_{\lambda \in \Lambda}\left(\gamma(\lambda) \vee \int_{0}^{\infty} e^{-tu} \left(\sup_{t \geq 0} q(f(t \wedge u), g(t \wedge u))\right) \dd u \right).
\end{equation}
It is well-established that $D([0, \infty); \mathbb{R})$ is a separable and complete metric space see, for example,~\cite[Theorem~5.6]{ethier-kurtz}.
In this paper, we use the following, well-known criterion for tightness in Skorokhod spaces. 
The first, from~\cite{ethier-kurtz}, has been slightly reformulated for out purposes. First, we define the following modulus of continuity: for $f \in D([0, \infty); \mathbb{R})$, $\eta > 0$, $T \in [0, \infty)$, we define 
\[
w'(f, \eta, T) := \inf_{\{t_i\}} \max_{i} \sup_{s, t \in [t_{i-1}, t_i)} \left|f(s) - f(t)\right|;
\]
where $\{t_i\}$ ranges over all partitions of $[0,T]$, such that $0 = t_{0} < t_{1} < \cdots < t_{n} = T$, with $t_{i+1} - t_{i} > \eta$ and $n \geq 1$.
\begin{lemma}[{\cite[Corollary~7.4, page 129]{ethier-kurtz}}] \label{lem:ethier-kurtz} 
A collection of probability measures $\left\{\mu_{n}\right\}_{n \in \mathbb{N}}$ on the metric space $D([0, \infty); \mathbb{R})$ is tight if and only if the following criteria are satisfied:
\begin{enumerate}
    \item For all 
    $t \in [0, \infty) \cap \mathbb{Q}$ and $\eps > 0$, there exists a compact set $K(t, \eps) \subseteq \mathcal{E}$ such that, for all $n \in \mathbb{N}$
    \begin{equation} \label{compact_cont}
    \liminf_{n \to \infty} \mu_{n}\left(\left\{f: f(t) \in K(t, \eps)\right\}\right) \geq 1- \eps ,
    \end{equation}
    \item For any $T \in [0, \infty)$, there exists $\eta > 0$ such that, for all $n \in \mathbb{N}$ 
    \begin{equation}\label{tight_cond2}
    \lim_{\eta \to 0} \limsup_{n \to \infty} \mu_{n}\left(\left\{f: w'(f, \eta, T) \geq \eps \right\} \right) = 0.
    \end{equation}
\end{enumerate}
\end{lemma}
\hfill $\blacksquare$ 

In literature surrounding stochastic processes, the first condition is often known as \emph{compact containment}. The following well-known tightness criterion due to Jakubowski applies more generally to $D([0, \infty); F)$, where $F$ is a completely regular Hausdorff spaces with metrisable compacts. Since we assume $\mathcal{E}$ is a metric space, it applies to $D([0, \infty); \mathcal{E})$:
\begin{lemma}[{\cite[Theorem~4.6]{jakubowski}}] \label{lem:jakubowski}
A collection of probability measures $\left\{\mu_{i}\right\}_{i \in I}$ on $D([0, \infty); \mathcal{E})$ is tight if and only if the following criteria are satisfied: 
\begin{enumerate}
    \item For any $t >0$ and $\eps > 0$ there is a compact set $K(t, \eps) \subseteq \mathcal{E}$ such that, for all $i \in I$, 
    \[\mu_{i}\left(\left\{f: \forall \, s \in [0,t] \, f(s) \in K(t, \eps)\right\}\right) \geq  1-\eps. \footnote{If $I = \mathbb{N}$ then we can replace this condition with $\liminf_{n \to \infty} \mu_{n}\left(\left\{f: \forall \, s \in [0,t] \, f(s) \in K(t, \eps)\right\}\right) \geq  1-\eps$, see, for example, the proof of~\cite[Corollary~7.4, page 130]{ethier-kurtz}.} 
    \]
    \item There exists a family of continuous functions $\mathbb{F}$ from $\mathcal{E}$ to $\mathbb{R}$ such that
    \begin{enumerate}
        \item The family $\mathbb{F}$ separates points, i.e., for any $x, y \in \mathcal{E}$ there exists $f \in \mathbb{F}$ such that $f(x) \neq f(y)$.
        \item The family $\mathbb{F}$ is closed under addition, i.e., if $f, g \in \mathbb{F}$ then $f+g \in \mathbb{F}$.
        \item Let, for $f \in \mathbb{F}$, $\tilde{f}: D([0, \infty); \mathcal{E}) \rightarrow D([0, \infty); \mathbb{R})$ denote the map such that $\tilde{f}(x) = f \circ x$, for $x\in~D([0, \infty);\mathcal{E})$. Then, for each $f \in \mathbb{F}$ the family of pushforward measures $\left\{\tilde{f}_{*}
        (\mu_{i})\right\}_{i \in I}$ is a tight family on $D([0, \infty); \mathbb{R})$.
    \end{enumerate}
\end{enumerate}
\end{lemma}
\hfill $\blacksquare$

{\bf Acknowledgements.}
We would like to thank Robert Patterson and Wolfgang K\"onig for some helpful discussions. 
This research has been partially funded by the
Deutsche Forschungsgemeinschaft (DFG) (project number 443759178) through grant SPP2265 ``Random Geometric Systems'', Project P01: `Spatial Coagulation and Gelation'. The authors acknowledge the support from Dipartimento di Matematica e Informatica ``Ulisse Dini'', University of Florence.  LA acknowledges partial support from~``Indam-GNAMPA''~Project~CUP\_E53C23001670001.

\bibliographystyle{abbrv}
\bibliography{ref}

\begin{thebibliography}{10}

\bibitem{aldous99}
D.~J. Aldous.
\newblock Deterministic and stochastic models for coalescence (aggregation and coagulation): a review of the mean-field theory for probabilists.
\newblock {\em Bernoulli}, 5(1):3--48, 1999.

\bibitem{AnIyMa24}
L.~Andreis, T.~Iyer, and E.~Magnanini.
\newblock Gelation in cluster coagulation processes.
\newblock arXiv preprint arXiv:2308.10232v2, 2024.

\bibitem{eibeck-wagner-01}
A.~Eibeck and W.~Wagner.
\newblock Stochastic particle approximations for {S}moluchoski's coagulation equation.
\newblock {\em Ann. Appl. Probab.}, 11(4):1137--1165, 2001.

\bibitem{ethier-kurtz}
S.~N. Ethier and T.~G. Kurtz.
\newblock {\em Markov processes}.
\newblock Wiley Series in Probability and Mathematical Statistics: Probability and Mathematical Statistics. John Wiley \& Sons, Inc., New York, 1986.
\newblock Characterization and convergence.

\bibitem{FeLuNoVe22}
M.~Ferreira, J.~Lukkarinen, A.~Nota, and J.~J.~L. Vel{\`a}zquez.
\newblock {Asymptotic localization in multicomponent mass conserving coagulation equations}.
\newblock arXiv preprint arXiv:2203.08076, 2022.

\bibitem{FeLuNoVe21}
M.~A. Ferreira, J.~Lukkarinen, A.~Nota, and J.~J.~L. Vel\'azquez.
\newblock Localization in stationary non-equilibrium solutions for multicomponent coagulation systems.
\newblock {\em Communications in Mathematical Physics}, 388(1):479--506, 2021.

\bibitem{fournier-giet-04}
N.~Fournier and J.-S. Giet.
\newblock Convergence of the {M}arcus-{L}ushnikov process.
\newblock {\em Methodol. Comput. Appl. Probab.}, 6(2):219--231, 2004.

\bibitem{Gil72}
D.~T. Gillespie.
\newblock The stochastic coalescence model for cloud droplet growth.
\newblock {\em J. Atmos. Sci.}, 29:1496--1510, 1972.

\bibitem{heydecker2019bilinear}
D.~Heydecker and R.~I.~A. Patterson.
\newblock Bilinear coagulation equations.
\newblock arXiv preprint arXiv:1902.07686, 2019.

\bibitem{jacquot-10}
S.~Jacquot.
\newblock A historical law of large numbers for the {M}arcus-{L}ushnikov process.
\newblock {\em Electron. J. Probab.}, 15:no. 19, 605--635, 2010.

\bibitem{jakubowski}
A.~Jakubowski.
\newblock On the {S}korokhod topology.
\newblock {\em Ann. Inst. H. Poincar\'{e} Probab. Statist.}, 22(3):263--285, 1986.

\bibitem{jeon98}
I.~Jeon.
\newblock Existence of gelling solutions for coagulation-fragmentation equations.
\newblock {\em Comm. Math. Phys.}, 194(3):541--567, 1998.

\bibitem{Lushnikov78}
A.~Lushnikov.
\newblock Some new aspects of coagulation theory.
\newblock {\em Izv. Acad. Sci. USSR, Ser. Phys. Atmos. Oceans}, 14, 01 1978.

\bibitem{Marcus68}
A.~H. Marcus.
\newblock Stochastic coalescence.
\newblock {\em Technometrics}, 10:133--143, 1968.

\bibitem{Norris}
J.~Norris.
\newblock {Measure solutions for the Smoluchowski coagulation-diffusion equation}.
\newblock arXiv preprint arXiv:1408.5228, 2014.

\bibitem{norris-coag-99}
J.~R. Norris.
\newblock {Smoluchowski's coagulation equation: uniqueness, nonuniqueness and a hydrodynamic limit for the stochastic coalescent}.
\newblock {\em The Annals of Applied Probability}, 9(1):78 -- 109, 1999.

\bibitem{norris-cluster-coag}
J.~R. Norris.
\newblock Cluster coagulation.
\newblock {\em Comm. Math. Phys.}, 209(2):407--435, 2000.

\bibitem{pollard}
D.~Pollard.
\newblock {\em Convergence of stochastic processes}.
\newblock Springer Series in Statistics. Springer-Verlag, New York, 1984.

\bibitem{revuzyor}
D.~Revuz and M.~Yor.
\newblock {\em Continuous martingales and {B}rownian motion}, volume 293 of {\em Grundlehren der mathematischen Wissenschaften [Fundamental Principles of Mathematical Sciences]}.
\newblock Springer-Verlag, Berlin, third edition, 1999.

\bibitem{rezakhanlou2013}
F.~Rezakhanlou.
\newblock Gelation for {M}arcus-{L}ushnikov process.
\newblock {\em Ann. Probab.}, 41(3B):1806--1830, 2013.

\bibitem{Thr23}
S.~Throm.
\newblock {Uniqueness of measure solutions for multi-component coagulation equations}.
\newblock arXiv preprint arXiv:2303.00775, 2023.

\bibitem{wagner-explosive-phenomena}
W.~Wagner.
\newblock Explosion phenomena in stochastic coagulation-fragmentation models.
\newblock {\em Ann. Appl. Probab.}, 15(3):2081--2112, 2005.

\end{thebibliography}

\end{document}